\documentclass[a4paper,reqno, 11pt]{amsart}%{article}
\pdfoutput=1
\usepackage[T1]{fontenc}
\usepackage{lmodern}
\usepackage[utf8]{inputenc}

\usepackage[left=2.8cm,right=2.8cm,top=1.6cm,bottom=2cm]{geometry}

\usepackage{amsmath}
\usepackage{amsfonts}
\usepackage{amssymb}
\usepackage{amsthm}
\usepackage{mathtools}
\usepackage{mathrsfs}
\usepackage{graphicx}
\usepackage{pgf,tikz}
\usetikzlibrary{arrows}
\usepackage{stmaryrd}
\usepackage{eqnarray}
\usepackage{csquotes}
\usepackage{relsize}
\usepackage{dsfont}
\usepackage{appendix}
\usepackage{soul}
\usepackage{bbm}
\usepackage{enumerate}
\usepackage{upgreek}
\usepackage{yhmath}

\definecolor{purple}{cmyk}{0.75,0.90,0,0}
\definecolor{DB}{rgb}{0.07,0.0,0.5}
\definecolor{DG}{rgb}{0.0,0.37,0.07}
\definecolor{DR}{rgb}{0.37,0,0.07}

\usepackage[
pdftitle={},%
pdfauthor={}, hyperindex=true, colorlinks=true, urlcolor={DR},
linkcolor={DR}, menucolor={DR}, citecolor={DR},
anchorcolor={DR},linktoc=all,pagebackref,final]{hyperref}

\makeatletter
\newtheorem*{rep@theorem}{\rep@title}
\newcommand{\newreptheorem}[2]{%
\newenvironment{rep#1}[1]{%
 \def\rep@title{#2 \ref{##1}}%
 \begin{rep@theorem}}%
 {\end{rep@theorem}}}
\makeatother

\newtheorem{theorem}{Theorem}[section]
\newreptheorem{theorem}{Theorem}
\newtheorem{coro}{Corollary}
\newtheorem{lemma}[theorem]{Lemma}
\newtheorem{proposition}[theorem]{Proposition}
\newreptheorem{proposition}{Proposition}
\newtheorem*{proposition*}{Proposition}

\newtheorem{claim}[theorem]{Claim}
\newtheorem{defi}{Definition}

\theoremstyle{remark}
\newtheorem{remark}[theorem]{Remark}

\theoremstyle{remark}
\newtheorem*{remark*}{Remark}

\numberwithin{equation}{section}

\theoremstyle{definition}

\theoremstyle{definition}

\newcommand{\fk}{\mathrm{FK}}
\newcommand{\is}{\mathrm{Is}}
\newcommand{\en}{\mathcal{E}}
\newcommand{\ber}[1]{\mathrm{Ber}_{#1}}
\newcommand{\cov}{\mathrm{Cov}}

\newcommand{\edag}{\en^\dagger}

\newcommand{\A}{\rotatebox[origin = c]{180}{$\forall$}}
\newcommand{\F}{{\mathcal F}}

\newcommand{\V}{{\mathcal V}}

\renewcommand{\P}{\mathbb P}
\newcommand{\R}{\mathbb R}

\newcommand{\1}{\mathbf 1}

\newcommand{\E}{\mathbb E}

\newcommand{\domi}{\preccurlyeq}

\def\rm{\reversemarginpar}

\title[Stochastic dominations for FK percolation and the Ising Model]{Stochastic dominations for FK percolation and sharp thinning thresholds for the Ising energy field}
\author{Paul Cahen}
\address{Universit\'e Claude Bernard Lyon 1, CNRS UMR 5208, Institut Camille Jordan, 69622 Villeurbanne, France}\email{cahen@math.univ-lyon1.fr}
\author{Avelio Sepúlveda}
\address{Universidad de Chile,  Centro de Modelamiento Matemático (AFB170001), UMI-CNRS 2807, Beauchef 851, Santiago, Chile.}\email{lsepulveda@dim.uchile.cl}

\begin{document}

\maketitle

\begin{abstract}
	At first glance, one would imagine that the energy field of the Ising model, the set of edges whose endpoints share the same spin, is stochastically monotone as a function of the coupling constants. However, this is not generally the case. In this paper, we introduce two weaker notions of stochastic domination that make this result true: $p$--weak and $p$--weak$^\dagger$ domination. Both of these notions depend on a parameter $p$ and we find the optimal values $p$ and $p^\dagger$ so that these dominations hold.
	
	One of the key ingredients in obtaining some of the results is a new stochastic domination relating FK percolations with different parameters $q,\tilde{q}\geq 1$ that is of independent interest.
	%\par\avelio{Along the way, we develop a parallel theory for the FKG property: we introduce some weakening of it and describe precisely which ones are satisfied in general for the energy graph of Ising.} 
\end{abstract}

\section{Introduction}
The Ising model has become a fundamental model in probability theory. For a given finite graph $G=(V,E)$ and coupling constants $J\in (\mathbb{R}^+)^{E}$, the \emph{(free) Ising model on $G$} is the probability distribution $\P_J$ on $\{\pm1\}^V$ satisfying
\begin{align*}
	\mathbb{P}_J(\upsigma)\propto e^{- H_J(\xi(\upsigma))}
\end{align*}
for any spin configuration $\upsigma: V\to \{\pm 1\}$, where the \emph{Ising energy field} $\xi(\upsigma)$ is the percolation configuration defined by $(\xi(\upsigma))_{xy} = \mathbf{1}_{\upsigma(x)=\upsigma(y)}$ for each edge $(xy)\in E$. Its energy is
\begin{align*}
	H_J(\xi) = -2\sum_{e \in E} J_{e} \xi_e. 
\end{align*}
In this paper, we will only deal with free boundary conditions.

Since $H_{J}(\xi)$ is minimized when $\xi$ is everywhere equal to $1$ (the fully aligned state), one would naturally expect that the law of $\xi(\sigma)$ stochastically increases in $J$. This expectation is correct at the level of one-edge marginals by the monotonicity of the spin correlations, but is false for the joint law. There are finite graphs for which increasing the coupling constants makes some increasing event in the energy field less likely, a phenomenon that persists even when $J$ is constant and increases across all edges \cite{haggstrom1996note}. The energy field also need not satisfy the FKG property (also known as positive association). Thus the failure is not a defect of the usual proof methods: it is a genuine obstruction to treating the full energy field as a monotone percolation model. In fact, as recently summarized by Klausen \cite{klausen2022monotonicity}, many natural observables of the Ising model fail to exhibit expected monotonicity and positive association. This issue is usually circumvented by the use of the monotonicity of the spin correlations, the FKG property of the spins, or by restricting the set of observables of the energy field as developed in \cite{cammarota1993stochastic}, which gives a generalized criterion for the observables that are indeed increasing in $J$.

The stochastic-order question for Ising spins has a complementary infinite-volume history. Liggett and Steif showed that plus states at distinct inverse temperatures are never stochastically comparable on $\mathbb Z^d$, whereas on regular trees they are ordered throughout the low-temperature regime \cite{liggett2006stochastic}. Later, Ray and Spinka proved that on every bounded-degree amenable graph, plus states at distinct inverse temperatures are incomparable, while on nonamenable graphs they are invariantly ordered at sufficiently low temperatures \cite[Theorem~1.3]{RS}. These results concern the vertex-spin field and the geometry of the underlying infinite graph.

The point of view of this paper is to measure the failure of monotonicity of the energy field rather than to bypass it. Given two laws $\mu$ and $\nu$ on edge configurations, suppose that the expected domination $\mu\domi\nu$ is false. We ask whether it
becomes true after applying the same independent thinning to both configurations: is there a \emph{strictly positive} parameter $p$ such that \begin{align*}
\mu\cap\ber{p}\domi \nu\cap\ber{p},
\end{align*}
a property that we call \emph{weak domination}, and what is the largest $p$ for which this holds? We also introduce an intermediate notion, weak$^\dagger$ domination, which requires weak domination for the set of open edges, and weak domination in the reverse direction for the set of closed edges: it is therefore a property which is symmetric upon taking the dual measure. Lying between classical stochastic domination and weak domination, it raises as well the question of the optimal parameter $p$ at which this weak$^\dagger$ domination holds. In both cases, the parameter $p$ is therefore not merely a technical loss: it measures how much monotonicity remains in a field whose full law is not monotone.

For the Ising energy field these questions have two exact answers. The first threshold is $p(J)=1-e^{-2J}$. At this value, the Edwards--Sokal coupling identifies $\en_J\cap\ber{p(J)}$ with the FK--Ising random-cluster model, which explains why this threshold is natural rather than accidental. It is the largest thinning parameter for which weak monotonicity and weak FKG hold uniformly over finite graphs. The second threshold is $p^\dagger(J)=1-e^{-4J}$. It appears when one asks for a version of the same statement that is also stable under complements, namely weak$^\dagger$ domination and weak$^\dagger$ FKG. The corresponding field $\edag_J$ has an abstract dual high-temperature expansion, see \eqref{eq formula for edag}, and is closely related to the dual percolation structures arising in the double-random-current couplings of \cite{AHL}. From the exact relations that it satisfies, see for example \eqref{eq formula edag general}, and the fact that it arises as a natural threshold, it seems natural to expect other interesting connections to the many models related to the Ising model, see \cite{hansen2025generalcouplingisingmodels} for a framework where these models are connected.

\medskip Parallel to our Ising results, we establish some (surprisingly unknown) stochastic dominations concerning FK percolation, a model prominent in probability theory also known as the random-cluster model. While FK percolation is a standard tool for studying Potts and Ising models on $\mathbb{Z}^d$ and other lattices (see \cite{grimmett2006random} for a comprehensive treatment), many of its core algebraic properties hold for arbitrary finite graphs and it is often studied over general graphs, see, for example, \cite{BGJ} for a detailed analysis of the FK model on the complete graph. To the best of our knowledge, our main FK result is the first to provide a stochastic inequality that ``mixes'' different $q$ parameters in an exact way, except for the well-known monotonicity in the $q$--parameter. Note that Grimmett obtained strict improvements of this monotonicity on lattices, which allowed him to prove strict monotonicity in $q$ of the critical point $p_c(q)$ of the $\fk$ percolation \cite{grimmett1995comparison}. As a special case of our inequality, we re-obtain a ``strong'' version of the monotonicity in $p$ (see \eqref{eq.strong_dom} below), which quantifies how much stochastically larger $\fk_{\tilde{p},q}$ is compared to $\fk_{p,q}$ for parameters $\tilde{p}>p$. In particular, it shows that the domination $\fk_{p,q}\domi\fk_{\tilde p,q}$ is both $\epsilon$-upwards and $\epsilon$-downwards-movable with an explicit $\epsilon=\epsilon(p,\tilde p,q)>0$, in the sense of \cite{broman2006refinements,broman2006dynamical}. The concept of movability was introduced to show the non-existence of exceptional times for two dynamical models outside criticality, and we note that, together with our work, it shows that dynamical FK percolation does not have exceptional time in the sub- and super-critical phases, a fact that to or knowledge was not noted before. 

However, during the preparation of the paper we discovered that the upward and downward movabilities for FK percolation were already present in the litterature: see Lemma 9.1 in \cite{lelli2024mixing} (or Lemma 2.5 in \cite{severo2024slab} for a less precise version). 

\medskip
The results of this paper therefore sit between two standard approaches. On the one hand, classical correlation inequalities such as FKG \cite{FKG,holley1974remarks} and the Ahlswede-Daykin four-function theorem \cite{ahlswede1978inequality}, GKS inequalities \cite{Griffiths,KS}, and GHS \cite{GHS} give robust monotonicity statements for spins, correlations, and related observables. On the other hand, modern coupling methods compare dependent fields with Bernoulli-type percolations, often relying on \cite{LSS} (see, e.g., \cite{dario-garban} in the context of the XY model or \cite{deijfen2024geometric} in the context of random geometry for applications of \cite{LSS}). Other methods have been fruitfully invented as well when \cite{LSS} does not apply directly, see e.g. \cite{BFO} or \cite{martineau2025stochastic}.

Here the full edge-energy field is kept, but the order relation is relaxed by an explicit independent thinning. The main results identify the exact amount of thinning needed, and the counterexamples show that these thresholds cannot be improved.

A particular case of this new stochastic domination is used to prove some of the results about weak monotonicity of the energy field of the Ising model, thus relating the two parts of the paper. We start by explaining our results on FK percolation.

\subsection{Strong monotonicity on FK percolation}
Let $G=(V,E)$ be a finite graph. We start by defining the FK percolation on $G$, introduced in \cite{FK-perco}, which generalizes Bernoulli percolation and depends on an additional global parameter $q>0$. For this, we view a percolation configuration $\omega\in\{0,1\}^E$ equivalently as a vector indexed by $E$ or as the subset of open edges $\{e\in E : \omega_e=1\}$. Letting $k(\omega)$ denote the number of connected components in the subgraph induced by $\omega$, the FK percolation model is defined as the probability measure on $\{0,1\}^E$ given by
\begin{align*}
	\fk_{p,q}(\omega) \propto q^{k(\omega)}\prod_{e\in E}p_e^{\omega_e}(1-p_e)^{1-\omega_e}.
\end{align*}
In this paper, we restrict our attention to the regime $q\ge 1$; for $q<1$, the model's behaviour changes fundamentally, and the results proven here no longer hold. Furthermore, for the special case $q=1$, we write $\fk_{p,1}=\ber{p}$ for Bernoulli percolation.

For $q\ge 1$, it is a well-known result that $\fk_{p,q}$ is stochastically increasing in the parameter $p$ and stochastically decreasing in the parameter $q$ (see, e.g., Theorem 3.21 in \cite{grimmett2006random}). Our first result generalizes and strengthens both of these monotonicities.

\begin{theorem}\label{main theorem introduction}
	Let $q\ge 1$ and let $p, \tilde{p} \in [0,1]^E$ be parameters such that $p \leq \tilde{p}$ pointwise. We have the following stochastic domination:
	\begin{align}\label{eq.strong_dom}
		& \fk_{p,q} \domi \fk_{\tilde{p},q} \cap \ber{p/\tilde{p}}, \qquad\qquad\qquad \text{\emph{(upwards-movability)}}
	\end{align}
	where, for two percolation measures $\mu$ and $\nu$, $\mu\cap\nu$ denotes the law of $\omega\cap\omega'$ for independent configurations $\omega\sim\mu$ and $\omega'\sim\nu$. The ratio $p/\tilde{p}$ is understood edgewise.
	
	More generally, for any real numbers $q, \tilde{q} \ge 1$ and parameters $p, \tilde{p} \in [0,1]^E$, we have
	\begin{align}
		\label{eq.mon_Intersection} \fk_{p\tilde{p}, q\tilde{q}} &\domi \fk_{p,q} \cap \fk_{\tilde{p}, \tilde{q}}, \qquad \text{and}\\
		\label{eq.mon_Union}  \fk_{\widehat{p}, q\tilde{q}} &\succcurlyeq \fk_{p,q} \cup \fk_{\tilde{p}, \tilde{q}},
	\end{align}
	where
	\begin{align}\label{eq. p hat **}
		\widehat{p}:=(p^*\tilde{p}^*)^* = \frac{p\tilde p+p\tilde{q}(1-\tilde{p})+\tilde{p}q(1-p)}{(1-p)(1-\tilde{p})+p\tilde p+p\tilde{q}(1-\tilde{p})+\tilde{p}q(1-p)}.
	\end{align}
	(The reason for the notation $(p^*\tilde{p}^*)^*$ will be made clear in Remark \ref{rk. union equivalent intersection}.)
\end{theorem}

Although the proof of this theorem relies only on a previously known argument, Holley's inequality (see Theorem \ref{theorem holley} below), the result was not discovered before. As mentionned before, the case $\tilde q=1$ had actually appeared in \cite{lelli2024mixing,severo2024slab}, in a slightly less general form.
%A weaker form of \eqref{eq.mon_Union} for the specific case of $\tilde{q}=1$ was recently obtained by Severo in the context of slab percolation for the Ising model (see Lemma 2.5 in \cite{severo2024slab}), using instead the Russo formula for FK percolation.

Furthermore, as discussed in Remark \ref{r.tight_ineq}, the bounds provided by Theorem \ref{main theorem introduction} are tight for all $q\ge 1$ in both the high-density $p\nearrow 1$ and low-density $p\searrow 0$ regimes.

\subsection{Weak monotonicities on the energy field}
For a spin configuration $\sigma\in\{\pm1\}^V$, we are primarily interested in the percolation configuration $\xi=\xi(\sigma)$, defined as the set of edges $e=(xy)$ such that $\sigma_x=\sigma_y$. More precisely, 
\begin{align*}
	\xi := \left\{e=(xy)\in E : \sigma_x=\sigma_y\right\} \quad \text{or equivalently} \quad \xi_{xy} := \mathbf{1}_{\sigma(x)=\sigma(y)} \quad \forall x\sim y.
\end{align*} 
When $\sigma$ has the law of an Ising model with coupling constants $J$, we say that $\xi(\sigma)\sim \en_J$ has the law of an Ising energy field on $G$ with coupling constants $J$. The energy field is a commonly studied observable of the Ising model itself; modern works have focused on planar domains at critical or near-critical parameters \cite{HS,izyurov2024energy,garban2025energy}. Let us mention as well \cite{ray2022finitary} that shows that the super-critical energy field is a ffiid (see their work for an introduction to ffiid processes). In our paper, we are more interested in the behaviour on a general graph at any temperature.

The relationship between the energy field and FK percolation is given by the Edwards–Sokal coupling (see \cite{edwards1988generalization}, or Theorem 1.13(b) in \cite{grimmett2006random} for a more modern approach). In our notation it says that

\begin{align}\label{equation edwards sokal}
	\fk_{p,2} = \en_J\cap\ber{p},
\end{align}
whenever $p_e=p(J)_e=1-e^{-2J_e}$ for every edge $e\in E$. Taking $q = 2$ in Theorem \ref{main theorem introduction} therefore gives the first weak monotonicity theorem for the Ising energy field.

\begin{theorem}[Weak stochastic monotonicity for the Ising energy field]\label{theorem weak stochastic monotonicity for Ising introduction}
	Let $J\leq \tilde{J}$ be (edge-dependent) coupling constants, where the inequality holds pointwise for every edge. Then, for any parameter $p \le 1-\exp(-2J)$,
	\begin{align}\label{eq.weak_dom}
		\en_{J} \cap \ber{p} \domi \en_{\tilde{J}} \cap \ber{p}.
	\end{align}
	Moreover, one cannot take $p>1-\exp(-2J)$ in general.
\end{theorem}

We refer to inequalities of the form \eqref{eq.weak_dom} as \emph{weak domination}; the formal definition and basic properties are given in Section \ref{sss.weak_stochastic_domination}. We are primarily concerned with the optimal parameter $p$ that can be chosen in the above inequality. As stated in the last sentence of Theorem \ref{theorem weak stochastic monotonicity for Ising introduction}, if one requires this inequality to hold in general, the maximum possible parameter $p$ for which the domination holds is $1-e^{-2J}$ (for a precise and rigorous formulation of this optimality, see Theorem \ref{theorem weak monotonicity and FKG for energy field} and Remark \ref{remark what does the assumption above p(J) mean?}). The domination \eqref{eq.weak_dom} itself, however, swiftly follows from the Edwards--Sokal coupling and Theorem \ref{main theorem introduction}:

\begin{proof}[Proof of \eqref{eq.weak_dom}]
	For $p=p(J)$, the left-hand side is exactly $\fk_{p(J),2}$ by equation \eqref{equation edwards sokal}. Similarly, by the Edwards--Sokal coupling, the right-hand side can be decomposed as \begin{align*}
		\en_{\tilde{J}}\cap\ber{p(\tilde{J})}\cap\ber{p(J)/p(\tilde{J})} = \fk_{p(\tilde{J}),2}\cap\ber{p(J)/p(\tilde{J})}.
	\end{align*} The stochastic domination is then a direct application of \eqref{eq.strong_dom}. The fact that it holds for smaller values of $p$ follows trivially by intersecting both sides with a further independent Bernoulli percolation of parameter $p/p(J)$. 
\end{proof}

Note that Theorem \ref{main theorem introduction} is essential to obtain \eqref{eq.weak_dom}: the classical stochastic monotonicity of $\fk_{p,2}$ in the parameter $p$ would instead rewrite as a stochastic domination where the two sides are intersected with Bernoulli percolations of \emph{different} parameters. Our strengthening of the FK-percolation monotonicity enables the use of the \emph{same parameter} $p$ on both sides.

\begin{remark}
	A completely analogous result to \eqref{eq.weak_dom} holds true for the $q$--Potts model for any integer $q\ge 2$, where the definition of the energy field and the proof are analogous to the case $q=2$. However, in this paper we concentrate on the Ising model, as our subsequent results are based on techniques developed specifically for it. Moreover, the energy field is especially interesting for the Ising model but far less so for larger values of $q$, since on a connected graph the spin configuration (up to a global flip) is measurable with respect to its energy field, a fact that is not true for the Potts model when $q\ge 3$.
\end{remark}

Theorem \ref{theorem weak stochastic monotonicity for Ising introduction} opens the door to a general study of weakened stochastic orders on probability measures on $\{0,1\}^E$. Weak domination asks for stochastic domination after a common independent thinning by a positive parameter. This turns out to be slightly stronger than the inequalities $\mu(A)\le\nu(A)$ for $A$ ranging over a specific subset of increasing events. See Section \ref{section general theory of domination} and Proposition \ref{proposition weak stochastic domination}. 

One may then define an in-between notion of stochastic domination. We say that \emph{$\mu$ is weakly$^\dagger$ dominated by $\nu$}, written $\mu\domi^{\dagger} \nu$, if at the same time $\mu$ is weakly stochastically dominated by $\nu$ (i.e. for some parameter $p>0$ the domination $\mu\cap\ber{p}\domi\nu\cap\ber{p}$ holds), and for some parameter $q<1$ the domination $\mu\cup\ber{q}\domi\nu\cup\ber{q}$ holds as well. In Section \ref{sss.weak_dagger_stochastic_domination} we define again this notion and we give a dual interpretation of the domination $\mu\cup\ber q\domi\nu\cup\ber q$.

Although in this paper we apply the notions of weak and weak$^\dagger$ domination only to the energy field, the general framework developed in Section \ref{section general theory of domination} is one of the model-independent components of the paper. It isolates a way of quantifying how much stochastic order remains after classical domination has failed: weak domination compares the two measures after a common Bernoulli thinning, while weak$^\dagger$ domination also imposes the dual comparison, obtained through a common Bernoulli sprinkling. We believe that these notions, and their analogues for FKG, may be useful in other contexts where exact monotonicity or positive association is absent, but where the models remain close to satisfying such properties.

\medskip
In the following theorem, we find the optimal threshold for weak$^\dagger$ domination of the energy field of the Ising model.
\begin{theorem}[Threshold for weak$^\dagger$ domination] \label{theorem weak * domination for Ising introduction}
	Let $J \leq \tilde{J}$ be (edge-dependent) coupling constants, where the inequality holds pointwise. Then, for any parameter $p\in[0,1]^E$ with $p<p^\dagger(J):=1-\exp(-4J)$ pointwise, we have
	\begin{align}\label{eq.weak_dom*}
		\en_{J} \cap \ber{p} \domi^\dagger \en_{\tilde{J}} \cap \ber{p}.
	\end{align}
	Furthermore, for $p\ge p^\dagger(J)$ the result does not hold in general.
\end{theorem}

As before, the last sentence of Theorem \ref{theorem weak * domination for Ising introduction} summarizes a sharpness statement. The precise sense in which $p^\dagger$ is the optimal threshold parameter is detailed in Theorem \ref{theorem weak *}, as explained in Subsection \ref{ss. at level pdag}. Note that in general one cannot take $p=p^\dagger(J)$ in \eqref{eq.weak_dom*}. This is not a contradiction since the weak$^\dagger$ domination is not closed under weak limits of measures.

To keep the introduction focused, we do not state the parallel results for the FKG property here. However, in Theorems \ref{theorem weak monotonicity and FKG for energy field} and \ref{theorem weak *}, we obtain not only the results for weak and weak$^\dagger$ stochastic domination, but also for the corresponding weak and weak$^\dagger$ FKG properties for the Ising energy field.

\subsection{Ideas of the proofs and organization}
The proofs have two separate components. The FK comparison Theorem \ref{main theorem introduction} is proved first, using Holley's inequality. Combined with the Edwards--Sokal identity, it gives the weak monotonicity result at the parameter $p(J)=1-e^{-2J}$, as explained above.

The weak$^\dagger$ threshold $p^\dagger(J)=1-e^{-4J}$ is obtained by a different argument. We introduce and study the model $\edag_J:=\en_J\cap\ber{1-e^{-4J}}$ and prove, in Proposition \ref{proposition inequalities for edag}, inequalities that are a slightly weaker form of weak$^\dagger$ domination at the threshold $p^\dagger(J)$. The proof uses a high-temperature expansion for an abstract Ising model associated with the cycle space of $G$. On planar graphs this is the usual dual high-temperature expansion; on general finite graphs it replaces the missing planar dual by an abstract Ising model (see Remark \ref{r.abstract_Ising} for a discussion on why this makes sense). This is the first step and the main difficulty in obtaining weak$^\dagger$ domination (and the parallel notion of weak$^\dagger$ FKG) below $p^\dagger$. A technical lemma then turns these endpoint inequalities into weak$^\dagger$ domination, and weak$^\dagger$ FKG, for every strictly smaller parameter.

The introduction of an abstract duality is also close in spirit to the argument we highlight in Remark \ref{rk. union equivalent intersection} to explain why one can see that the two statements of Theorem \ref{main theorem introduction} about FK percolation are equivalent.

Optimality is proved by explicit counterexamples. For weak$^\dagger$ domination, a cycle with four vertices already shows that parameters larger than $p^\dagger(J)$ cannot be allowed in general. For weak domination above the Edwards--Sokal threshold, the counterexample is chosen to be a cycle as well, though it needs to be quite large. The argument uses a decomposition of the energy field into a linear combination of Bernoulli measures, including one with a parameter outside $[0,1]$, which makes it possible to detect the failure of monotonicity for a suitable increasing event. 

\medskip The paper is organized as follows. Section \ref{section FK percolation} proves the FK comparison Theorem \ref{main theorem introduction}. Section \ref{section general theory of domination} introduces weak domination, weak$^\dagger$ domination, and the corresponding weak FKG notions, together with the general facts used later. Section \ref{section Ising's energy field and fk hat} proves the Ising results, develops the abstract high-temperature expansion for $\edag_J$, and gives the counterexamples establishing optimality of the two thresholds.

\subsection*{Acknowledgements} We wish to thank Nathan de Montgolfier, Lorca Heeney, Émile Averous and Tomás Alcalde-López for useful discussions. Thanks as well to Romain Panis for having encouraged us at an early stage of the project. Many thanks to Christophe Garban for great feedback on an earlier version of this paper, and to Yinon Spinka for providing us useful references.

AI tools contributed at the places explicitly indicated in the text. In particular, the decomposition \eqref{eq decomposition en_J in sum of bernoulli} arose in an AI-assisted discussion; despite our search, we may have missed an earlier source and thus failed to credit its authors.

Both authors are supported by ERC 101043450 Vortex. 
The research of A.S. was also supported by Centro de Modelamiento Matem\'{a}tico Basal Funds FB210005 from ANID-Chile and by Fondecyt Grant 1240884.

\section{A new inequality for the FK-model}
\label{section FK percolation}
The objective of this section is to prove Theorem \ref{main theorem introduction}, that is, to show that the intersection of two independent FK-models stochastically dominates the FK-model whose parameters are the products of the former ones. 

Let us start by recalling the definition of the FK-model. Fix a finite graph $G=(V,E)$, a real number $q>0$, and an assignment $p\in [0,1]^E$ of a number in $[0,1]$ to each edge. The FK-model (with free boundary conditions) is the probability measure on percolation configurations $\omega\in \mathcal P(E)=\{0,1\}^E$
\begin{align*}
	\fk_{p,q}(\omega)\propto \prod_{e\in \omega} p_e \prod_{e\in E\backslash \omega}(1-p_e) \ q^{k(\omega)},
\end{align*}
where $k(\omega)$ is the number of connected components of the graph $G_\omega= (V,\omega)$. In the following, we always identify $\omega$ either with a subset of $E$ or with an element of $\{0,1\}^E$; in this context, for an edge $e\in E$, the statement $e\in \omega$ is equivalent to $\omega(e)=1$.

We can now state precisely the two main results of this section. To do this, recall that for any two probability measures $\mu,\nu$ on $\mathcal P(E)$, we define $\mu\cap \nu$ as the push-forward under the intersection of the measure $\mu\times \nu$. In words, it is the law of the intersection of two independent configurations $\omega_1\sim \mu$ and $\omega_2\sim\nu$. 
\begin{theorem}\label{th. intersection}
    For any parameters $p,\tilde{p}\in [0,1]^{E}$ and real numbers $q,\tilde{q}\ge1$, we have the following stochastic domination:
    \begin{equation}\label{eq. stochastic domination intersection}
        \fk_{p,q} \cap \fk_{\tilde{p},\tilde{q}} \succcurlyeq \fk_{p\tilde{p},q\tilde{q}}.
    \end{equation}
\end{theorem}

Defining the union $\mu\cup\nu$ of two probability measures $\mu,\nu$ in the same way, one gets a complementary (and in fact equivalent, see Remark \ref{rk. union equivalent intersection}) stochastic domination for the union of two FK measures. 

\begin{theorem}\label{th. union}
 For any parameters $p,\tilde{p}\in [0,1]^{E}$ and real numbers $q,\tilde{q}\ge1$, we have the following stochastic domination:
    \begin{equation}\label{eq. stochastic domination union}
        \fk_{p,q} \cup \fk_{\tilde{p},\tilde{q}} \domi \fk_{\widehat{p},q\tilde{q}}
    \end{equation}
where $\widehat{p}$ is given by the formula \eqref{eq. p hat **}.
\end{theorem}

The proofs of these two theorems are very similar and are based on Holley's inequality, which we recall below. First, we need to introduce the following quantities. Let $\mu$ be a probability measure on $\{0,1\}^E$, let $e\in E$ be an edge, and let $\xi\in\{0,1\}^{E\backslash\{e\}}$ be a percolation configuration outside of $e$, and define \begin{align*} \mu(e\mid\xi):=\mu\left(e\in \omega \ \big| \ \omega|_{E\backslash \{e\}}=\xi\right).
\end{align*}
This is the probability that a given edge belongs to $\omega$ given the sample outside the edge.
The following result, known as Holley's inequality, has proven to be very useful to show stochastic dominations in many contexts. Although we do not give a proof, let us say that it can be deduced quickly from a coupling of two Markov chains.

\begin{theorem}[Holley's inequality, \cite{holley1974remarks}]\label{theorem holley}
Let $\mu,\nu$ be measures of full support on $\{0,1\}^E$. Let us assume that for any $e\in E$ and $\xi,\xi'\in \{0,1\}^{E\backslash \{e\}}$ with $\xi'\le\xi$, we have the following inequality \begin{equation}\label{Holley condition}
\mu(e \mid \xi') \le \nu(e \mid \xi).
\end{equation}
\noindent Then $\mu\domi\nu$.
\end{theorem}

In order to make use of this criterion for stochastic domination, we want to know what the values $\mu(e\mid\xi)$ are when $\mu$ is an FK-percolation measure. The answer is given by the following proposition:

\begin{proposition}[One-point function conditioned on the exterior, Theorem 3.1 in \cite{grimmett2006random}]\label{proposition one-point function}
For $q> 0$ and $p\in [0,1]^E$, the one-point function of the FK-model conditionally on the outside is given by
\begin{equation}\label{equation one-point function}
\fk_{p,q}(e\mid \xi) = \begin{cases} 
\quad \ p_e & \text{ if } \xi \text{ connects the two endpoints of }e, \\
 \frac{p_e}{p_e+q(1-p_e)} &\text{ if } \xi \text{ does not.}
\end{cases}
\end{equation}
Furthermore, this property defines $\fk_{p,q}$.
\end{proposition}

\begin{remark}\label{r.tight_ineq} Let us now use Proposition \ref{proposition one-point function} to see that Theorem \ref{th. intersection} is sharp at least when both $p,\tilde{p}\searrow 0$ or both $p,\tilde{p}\nearrow 1$. For the first case, note that when $p\searrow 0$, w.h.p. $\omega|_{\{e\}^c}=0$ and therefore 
	\begin{align*}
	\fk_{p,q}(\omega_e)\sim\frac{p_e}{p_e+(1-p_e)q}(1+o(1))\sim\frac{p_e}q,
	\end{align*}
	which achieves asymptotic equality for \eqref{eq. stochastic domination intersection}.
	
	For the regime close to $1$, take $p\nearrow 1$, and note that w.h.p. $\omega_{\{e\}^c}=1$, and therefore
	\begin{align*}
		\fk_{p,q}(\omega_e)\sim p_e(1+o(1))\sim p_e,
	\end{align*}
	which also achieves asymptotic equality for \eqref{eq. stochastic domination intersection}. A similar reasoning can be done to show the sharpness of Theorem \ref{th. union}.
\end{remark}

\medskip

We now have all the tools to prove Theorems \ref{th. intersection} and \ref{th. union}. The two proofs are the same, though the computations are a bit more tedious for the proof of Theorem \ref{th. union}, as one can guess from the form of $\widehat{p}$. Since we do not use the union statement, Theorem \ref{th. union}, in the rest of the paper, we first prove the intersection statement, Theorem \ref{th. intersection}, carefully and then explain how to adapt the proof to Theorem \ref{th. union}.
\par We explain in Remark \ref{rk. union equivalent intersection} how one can see that the two statements are actually equivalent, and give a sense of why $\widehat{p}$ is given by the complicated expression \eqref{eq. p hat **}. We actually use this viewpoint to streamline some of the computations not performed in the proof of Theorem \ref{th. union}. 

\begin{proof}[Proof of Theorem \ref{th. intersection}]
We will show that the measures $\mu:=\fk_{p\tilde{p},q\tilde{q}}$ and $\nu:=\fk_{p,q}\cap\fk_{\tilde{p},\tilde{q}}$ jointly satisfy the assumptions of Theorem \ref{theorem holley}, which is of course enough. We may assume that the parameters $p,\tilde p$ are strictly between $0$ and $1$; the limiting cases follow by taking weak limits. Therefore, the distributions $\mu,\nu$ are of full support, so we focus on \eqref{Holley condition}.
	
	Let us first note that for all $q\ge 1$ and $p\in [0,1]^E$, the function $\fk_{p,q}(e\mid\cdot)$ is nondecreasing in the outside percolation $\xi$. This follows from the fact that the event that $e^+$ is connected to $e^-$ using $\xi$ is increasing in $\xi$ and that $p_e\ge p_e/(p_e+q(1-p_e))$, as $q\geq 1$.
	
	From the monotonicity of the function $\fk_{p,q}(e\mid\cdot)$, we note that it is enough to show Holley's inequality \eqref{Holley condition} restricted to the case $\xi=\xi'$. Namely, for any $e\in E$ and $\xi\in \{0,1\}^{E\backslash \{e\}}$
	\begin{align}\label{e.Holley_equality}
		\nu(e\mid \xi) \geq \mu(e\mid\xi)=\fk_{p\tilde{p},q\tilde{q}}(e\mid\xi).
	\end{align}
	To do this, let us lower bound the left-hand side.
\begin{claim}\label{c.Holley_Equality}If $\xi$ connects the two endpoints of $e$, $\nu(e\mid\xi)=p\tilde{p}$. If $\xi$ does not connect the two endpoints of $e$, $\nu(e\mid\xi)\ge \frac{p}{p+q(1-p)}\frac{\tilde{p}}{\tilde{p}+\tilde{q}(1-\tilde{p})}$.
\end{claim}
We now assume the claim and show \eqref{e.Holley_equality}. In this case, it is enough to show that \begin{align}\label{e. two ineq}
	&p\tilde{p}  \ge  p\tilde{p} \hspace{0.05\textwidth}\text{ and }\hspace{0.05\textwidth}
	\frac{p}{p+q(1-p)}\frac{\tilde{p}}{\tilde{p}+\tilde{q}(1-\tilde{p})}  \ge  \frac{p\tilde{p}}{p\tilde{p}+q\tilde{q}(1-p\tilde{p})}.
\end{align}
Of course, the first inequality is trivial; the second can be shown to be equivalent to \begin{align*}
	q(1-p)(\tilde{q}-1)\tilde{p} + (q-1)p\tilde{q}(1-\tilde{p}) \ge 0,
\end{align*}
which is true for all values of $p,\tilde p\in [0,1]$ and $q,\tilde q\geq 1$.
\end{proof}
It remains only to prove the claim.

\begin{proof}[Proof of Claim \ref{c.Holley_Equality}]
	Let us denote by $\E_\xi$ the expectation with respect to $\fk_{p,q}\otimes\fk_{\tilde{p},\tilde{q}}$ conditioned on the event \[\left\{\big(\omega\cap\tilde\omega\big)|_{\{e\}^c}=\xi\right\}.\]
	By definition, $\nu(e\mid\xi)=\E_\xi[\omega_e\tilde\omega_e]$, therefore \begin{align}
		\nu(e\mid\xi)  =  \E_\xi\Big[\fk_{p,q}(e\mid\omega|_{\{e\}^c})\fk_{\tilde{p},\tilde{q}}(e\mid\tilde\omega|_{\{e\}^c})\Big].
	\end{align}
	However, since $q,\tilde{q}\ge1$, \begin{align}
		\fk_{p,q}(e\mid\omega|_{\{e\}^c})  \ge  \frac{p}{p+q(1-p)}\qquad a.s. \\
		\fk_{\tilde{p},\tilde{q}}(e\mid\tilde\omega|_{\{e\}^c})  \ge  \frac{\tilde{p}}{\tilde{p}+\tilde{q}(1-\tilde{p})}\qquad a.s.
	\end{align}
	and therefore the right-hand side is at least $\frac{p}{p+q(1-p)}\frac{\tilde{p}}{\tilde{p}+\tilde{q}(1-\tilde{p})}$. When $\xi$ connects the endpoints of $e$, however, both $\omega|_{\{e\}^c}$ and $\tilde\omega|_{\{e\}^c}$ also connect the endpoints of $e$, $\E_\xi$-a.s. Therefore, in this case, 
	\begin{align}
		\fk_{p,q}(e\mid\omega|_{\{e\}^c})  =  p \qquad a.s. \\
		\fk_{\tilde{p},\tilde{q}}(e\mid\tilde\omega|_{\{e\}^c})  =  \tilde{p} \qquad a.s.
	\end{align}
	whence the claim.
\end{proof}

\begin{proof}[Proof of Theorem \ref{th. union}]
The proof of the union statement closely follows the proof presented above for \eqref{eq. stochastic domination intersection}. We define $\mu=\fk_{p,q} \cup \fk_{\tilde{p},\tilde{q}}$ and $\nu=\fk_{\widehat{p},q\tilde{q}}$ and we apply Holley's theorem (Theorem \ref{theorem holley}) to the pair $\mu,\nu$. Of course, as in the proof of the intersection statement, we may assume that $p,\tilde p$ are always strictly between 0 and 1, so that the measures $\mu,\nu$ are of full support. We may thus focus on proving Holley's inequality \eqref{Holley condition}.
\par The relevant claim is then 
\begin{claim}\label{c.Holley_Equality_2}Let $e\in E$ be any edge, and $\xi$ be a configuration on the set $E\setminus\{e\}$. If $\xi$ does not connect the two endpoints of $e$, $\mu(e\mid\xi)=\nu(e\mid \xi)$. If $\xi$ does connect the two endpoints of $e$, $\mu(e\mid\xi)\le p+\tilde p -p\tilde p$.
\end{claim}
This claim is proven similarly to Claim \ref{c.Holley_Equality}: for the first point, one must have $1-\mu(e\mid\xi)=(1-\frac p{p+q(1-p)})(1-\frac{\tilde{p}}{\tilde{p}+\tilde q(1-\tilde p)})$ so the equality follows from $\nu(e\mid\xi)=\frac{\widehat{p}}{\widehat{p}+q\tilde q(1-\widehat{p})}$ and the additional relation 

\begin{align}\label{eq. relation p**}
\left(1-\frac p{p+q(1-p)}\right)\cdot\left(1-\frac{\tilde{p}}{\tilde{p}+\tilde q(1-\tilde p)}\right)=1-\frac{\widehat{p}}{\widehat{p}+q\tilde q(1-\widehat{p})}.
\end{align}
(This equality follows directly in Remark \ref{rk. smooth computations with p **} using the relation $\widehat{p}=(p^*\tilde p^*)^*$ explained in Remark \ref{rk. union equivalent intersection}.)

\medskip For the second point, the same reasoning as in the proof of Claim \ref{c.Holley_Equality} gives $1-\mu(e\mid\xi)\geq (1-p)(1-\tilde p)$, that is $\mu(e\mid\xi)\le p+\tilde p -p\tilde p$. This shows the claim.

\medskip Furthermore, this claim can be used to prove \eqref{e.Holley_equality} since in the case $e^-\stackrel{\xi}{\not\leftrightarrow}e^+$ we have $\mu(e\mid\xi)=\nu(e\mid \xi)$ and in the case $e^-\stackrel{\xi}{\leftrightarrow}e^+$ we have $\mu(e\mid\xi)\le\nu(e\mid \xi)$ since \begin{align*}
p+\tilde p -p\tilde p \le \widehat{p}.
\end{align*}
(Once again, we show this last inequality in a meaningful way in Remark \ref{rk. smooth computations with p **}.)

Hence, if $e\in E$ is an edge and $\xi'\le\xi$ are configurations on the set $E\setminus\{e\}$, then $\mu(e\mid\xi')\le\nu(e\mid\xi')\le\nu(e\mid\xi)$ (the second inequality relies on the monotonicity of $\nu(e\mid \cdot)$ since $\nu$ is a FK measure) so that Holley's stochastic domination criterion applies, and we obtain the stochastic domination \eqref{eq. stochastic domination union}.
\end{proof}

\begin{remark}\label{rk. union equivalent intersection}
It can actually be seen that the two stochastic dominations are equivalent in a way that explains the reason why we wrote $\widehat{p}$ in the form $(p^*\tilde{p}^*)^*$. Indeed, when $G$ is planar, if $p^*:=\frac{q(1-p)}{p+q(1-p)}$, $\tilde{p}^*:=\frac{\tilde q(1-\tilde p)}{\tilde p+\tilde q(1- \tilde p)}$ and $(p^*\tilde p^*)^*:=\frac{q\tilde q(1-p^* \tilde p^*)}{p^*\tilde p^*+q\tilde q(1-p^* \tilde p^*)} = \widehat{p}$, then by the standard fact that the dual of $\fk_{p,q}$ percolation is $\fk_{p^*,q}$ percolation on the dual graph with dual parameter $p^*=\frac{q(1-p)}{p+q(1-p)}$ (see, e.g., Section 6.1 in \cite{grimmett2006random}), the domination \eqref{eq. stochastic domination union} is precisely the domination \eqref{eq. stochastic domination intersection} reversed where we take the dual edges. Note that the map $\cdot\mapsto \cdot^*$ \emph{depends on the parameter} $q$, and that in the formula $\widehat{p}=(p^*\tilde{p}^*)^*$, we actually use three such maps, namely for the three parameters $q,\tilde q,q\tilde q$.
\par It is actually possible in the case of a general graph $G$ to define a somewhat more general FK percolation for which Theorem \ref{th. intersection} still holds (with precisely the same proof) and which allows one to take complements even in the absence of a dual graph, so that the equivalence between Theorem \ref{th. intersection} and \ref{th. union} is always true. We do not pursue this further, but a similar theory is developed for the Ising model in Subsection \ref{section abstract Ising}, where we introduce an abstract Ising model allowing some duality even in the case of a non-planar graph.
\end{remark}

\begin{remark}\label{rk. smooth computations with p **}
Let us explain why \eqref{eq. relation p**} holds, without simply plugging the value given by \eqref{eq. p hat **}. Indeed, in the notations of the previous remark, it becomes $p^*\tilde p^*=(\widehat p)^*$ (where the $^*$ are respectively for the parameters $q,\tilde q,q\tilde q$), which is precisely our definition of $\widehat{p}$ (since taking the dual parameter is an involution regardless of the $q$--parameter).
\par Now, let us explain where $p+\tilde{p}-p\tilde p\le\widehat{p}$ comes from. First, rewrite it as $(1-p)(1-\tilde p)\ge 1-\widehat{p}$. Second, since $p,p^*$ are related by $1-p=\frac{p^*}{p^*+q(1-p^*)}$ (and similarly for $\tilde{p}, \tilde{p}^*$ and $\widehat{p},p^*\tilde p^*$), the inequality is precisely the right-hand side of \eqref{e. two ineq} but with $p,\tilde p$ replaced by $p^*,\tilde{p}^*$, and so is true by the same rearrangement. 
\par These two justifications once again show that what we really use is (the proof of) the intersection statement for the dual parameters, as explained in the previous remark.
\end{remark}

\section{Variants of stochastic domination and FKG inequality}
\label{section general theory of domination}
Motivated by the fact that the energy field of the Ising model is not stochastically monotone with respect to the temperature, we introduce, in this section, two different versions of stochastic domination. We will see afterwards, in Section \ref{section Ising's energy field and fk hat}, that, in fact, the energy field of the Ising model is monotone in the temperature in these weaker senses in a very precise way.

The main idea of these weaker variants of stochastic domination is to hide the anomalies that do not allow the domination via the intersection with an independent Bernoulli percolation. Thus, to simplify the notation, for any probability measure $\mu$ and parameter $p\in[0,1]^E$, we write $\mu_p=\mu\cap\ber{p}$. We also denote by $\mu^c$ the law of $\omega^c$ where $\omega\sim\mu$ and we call it the \emph{dual measure} of $\mu$.

In addition, some specific events in this section play a key role. For $F\subset E$, we denote by $\forall_F$ and $\A_F$ the events that all edges $e\in F$ are open and closed, respectively, for the percolation configuration $\omega\in\{0,1\}^E$. More precisely
\begin{align*}
\forall_F:=\{\omega\ |\ F\subset\omega\}\quad \text{ and }\quad
\A_F:=\{\omega \ |\ \omega\cap F=\emptyset\}.
\end{align*}
The notation $\forall_F$ comes from the fact that we ask \emph{all} edges in $F$ to be open, and the notation $\A$ is simply the $\pi$-rotation of $\forall$. Keeping in mind that the events $\forall_F$ and $\A_F$ are somehow dual of each other, one sees that \begin{equation}\label{eq. duality of events measures}
	\mu^c(\A_F)=\mu(\forall_F).
\end{equation}

We start this section by introducing and discussing two variants of stochastic domination. Then, we discuss the weakening of the FKG properties in this context. We end by discussing a natural strengthening of the stochastic domination.

\par

\subsection{Weakening of stochastic domination}

As stated before, we introduce two notions of stochastic domination. The weaker one is called \textit{weak stochastic domination}, and the stronger one is called \textit{weak$^\dagger$ stochastic domination}.

\subsubsection{Weak stochastic domination.} \label{sss.weak_stochastic_domination} 
Let us start by introducing the weakest form of stochastic domination. 
\begin{defi}Take two probability measures $\mu,\nu$ on $\{0,1\}^E$ and $p\in (0,1]^{E}$. We say that the measure $\mu$ is $p$--dominated by $\nu$ if $\mu_p\domi\nu_p$. We denote this relationship $\mu\domi_p\nu$. Furthermore, we say that $\mu$ is weakly dominated by $\nu$ if there exists a parameter $p\in (0,1]^E$ such that $\mu$ is $p$--dominated by $\nu$.
\end{defi}

\par As a first observation of this new notion, let us note that if two probability measures $\mu,\nu$ on $\{0,1\}^E$ are stochastically ordered, that is $\mu \domi \nu$, then for any parameter $p \in (0,1]^E$ one also has $\mu_p \domi \nu_p$. The converse, however, fails in general: the existence of a parameter $p$ with strictly positive coordinates such that $\mu_p \domi \nu_p$ does \emph{not} imply that $\mu \domi \nu$\footnote{An explicit counterexample is provided, e.g., in the proof of Theorem \ref{theorem weak stochastic monotonicity for Ising introduction}, Subsection \ref{section above p(J)}}. Consequently, this notion is strictly weaker than classical stochastic domination, and becomes weaker as $p$ decreases.
\par The reason we are interested in this definition is that we want to find the best parameter $p$ for which the energy field of the Ising model (defined in the Introduction and in Section \ref{section Ising's energy field and fk hat}) is $p$--monotone, in other words, the largest $p$ such that $\en_{J,p}$ is monotone in $J$.

We now list elementary properties of the weak stochastic domination.
\begin{proposition}\label{proposition weak stochastic domination}
Let $\mu,\nu$ be probability measures on $\{0,1\}^E$. \begin{enumerate}[(i)]
\item if $p\le \tilde{p}$ are two parameters, and $\mu\domi_{\tilde{p}}\nu$, then $\mu\domi_{p}\nu$. 
\item if $\mu$ is weakly dominated by $\nu$, then $\mu(\forall_F)\le\nu(\forall_F)$ for all $F\subset E$.
\item if for any $F\subset E$, $\mu(\forall_F)\le\nu(\forall_F)$, and the inequality is strict for any nonempty $F$, then $\mu$ is weakly dominated by $\nu$.
\end{enumerate}
\end{proposition}

\begin{proof}
First, note that (i) follows as $\mu_{p}=\mu_{\tilde{p}}\cap\ber{p/\tilde{p}}$ and the same is true for $\nu_{p}$, whence the result. To prove (ii), we use that $\mu_p(\forall_F)=\mu(\forall_F)\prod_{e\in F}p_e$ and that the analogue is true for $\nu_p$. The result follows by dividing by $\prod_{e\in F}p_e>0$.

Finally, to prove (iii), we need to find $p\in (0,1)^E$ such that for any increasing event $A$, $\mu_p(A)\le\nu_p(A)$. As $\mathcal P(E)$ is finite, it is enough to show that for all increasing sets $A$ there is a constant parameter $p>0$ such that $\mu_p(A)\leq \nu_p(A)$.

If $A$ is either $\forall_\emptyset$ or $\forall_E$, the conclusion is part of the assumption. In the other cases, $A$ can be written as \begin{align*}
A = \bigcup_F \forall_F
\end{align*}
where $F$ ranges over the set of minimal (for inclusion) witnesses of $A$. Define $m=m(A)$ as the minimal number of edges that such a witness can have. One can see that as $p\searrow 0$ \begin{equation}\label{equation mup(A) for small p}
\mu_p(A)=p^m \sum_{F}\mu(\forall_F) +o(p^m)
\end{equation} where the sum ranges over all witnesses $F$ of cardinality $m$. A similar decomposition holds for $\nu_p(A)$ and by assumption, since the first non-zero order term is strictly smaller for $\mu$ than for $\nu$, one has $\mu_p(A)<\nu_p(A)$ for all $p>0$ sufficiently small.
\end{proof}

\begin{remark}\label{remark counter example reciprocal weak domination}
If the inequalities $\mu(\forall_F)\le\nu(\forall_F)$ in the third item are not required to be strict, the implication of weak domination is actually false as shown by the following example: take $E=\{1,2\}$, $ \mu=\ber{1/2}$ and $\nu=\frac12\delta_{(0,0)}+\frac12\delta_{(1,1)}$, the measure where one tosses one fair coin to decide if all edges are open or closed. One can check that $\mu(\forall_F)\le\nu(\forall_F)$ for any $F\subseteq E$, but for every positive parameter $p$, $\mu_p(A)>\nu_p(A)$ where $A$ is the increasing event $\{\text{at least one edge is open}\}$.
\end{remark}

\subsubsection{Weak$^\dagger$ stochastic domination}\label{sss.weak_dagger_stochastic_domination}
In this section, we introduce a notion of stochastic domination that lies between weak and classical stochastic domination. To do this, it is useful to introduce the notion of dual weak stochastic domination.

\begin{defi}[Dual weak stochastic domination] For probability measures $\mu, \nu$ on $\{0,1\}^E$ and a parameter $q\in[0,1)^E$, we say that $\nu$ dually $q$--weakly dominates $\mu$, a relation denoted $\mu\domi^c_q \nu$, if $\nu^c \domi_{1-q} \mu^c$. If there exists such a parameter $q$\footnote{Note that we ask that $q_e<1$ for every edge $e$}, we say that $\nu$ dually weakly dominates $\mu$.
\end{defi}

Note that this dual notion of domination may be understood as follows. The relation $\mu\domi^{c}_q\nu$ is equivalent to the domination $\mu\cup\ber{q}\domi\nu\cup\ber{q}$. Note that the second point of Proposition \ref{proposition weak stochastic domination} implies the inequalities $\mu(\A_F)\ge\nu(\A_F)$ for all $F\subset E$, and the third point implies that in case of strict inequalities for all $\emptyset\neq F\subset E$, we have dual weak domination.

We can now define the concept of weak$^\dagger$ stochastic domination.

\begin{defi}[Weak$^\dagger$ stochastic domination]\label{d.weak_dagger_stochastic_domination}If $\mu,\nu$ are two probability measures on $\{0,1\}^E$, we say that $\nu$ weakly$^\dagger$ dominates $\mu$ if there are parameters $p\in (0,1]^E$ and $q\in[0,1)^E$ such that $\mu\domi_p \nu$ and $\mu\domi^c_q \nu$. We denote this relation $\mu\domi^\dagger\nu$, and we add $p$ and $q$ as sub-indices if needed.
\end{defi}

The reason why we introduce this concept is the following: returning to the situation where we do have a weak domination between measures $\mu,\nu$ and we are interested in the largest parameter for which stochastic domination holds, we can also investigate the largest parameter $p$ for which $\mu_p\domi^\dagger\nu_p$, a weaker property, that might give rise to a different and interesting threshold. This is the case when $\mu$ and $\nu$ are the energy fields of the Ising model at different temperatures as we will see in Section \ref{section Ising's energy field and fk hat}.

Let us now check\footnote{The proof we present for this proposition was found in a discussion with Gemini Pro version 3.1} that the weak$^\dagger$ domination also satisfies the first point of Proposition \ref{proposition weak stochastic domination}.
\begin{proposition}\label{proposition weak * domination stable}
Let $\mu,\nu$ be probability measures on $\{0,1\}^E$ such that $\mu\domi^{\dagger}_{p,q}\nu$. Then, for any parameter $r\in (0,1]^E$, $\mu_r\domi^{\dagger}_{p',q'}\nu_r$, with $p',q'$ defined by $p'=1\wedge (p/r)$ and $q'=qr'$. Here, $r'=r/(1-q(1-r))$.
\par In particular, weak$^\dagger$ domination is stable under intersection with $\ber{r}$.
\end{proposition}

\begin{proof}
First, it follows directly that $\mu\domi_{p'r} \nu$, i.e. $\mu_r\domi_{p'}\nu_r$ by monotonicity (first bullet point of Proposition \ref{proposition weak stochastic domination}, since $p'r\le p$). To show that $\mu_{r}\domi^{c}_{q'} \nu_r$, we need the following claim.
	\begin{claim}\label{claim union intersection} 
		With the notation of the proposition, we have the following equality of measures:
		\begin{equation}\label{eq union intersection}
			\big(\mu\cap\ber{r}\big)\cup\ber{q'} = \big(\mu\cup\ber{q}\big)\cap\ber{r'}.
		\end{equation}
	\end{claim}
	Let us first see how the claim implies that $\mu_{r}\domi^c_{q'} \nu_r$. This follows from computing
	\begin{align*}
		\mu_r \cup \ber{q'} = (\mu \cup \ber{q}) \cap \ber{r'}\domi (\nu \cup \ber{q}) \cap \ber{r'} = \nu_r \cup \ber{q'}.
	\end{align*}
	To finish this proof, we now only need to prove the claim.	
	\begin{proof}[Proof of Claim \ref{claim union intersection}]
		As we can first condition on $\xi\sim \mu$, it is enough to prove the claim for $\mu=\delta_{\xi}$ for every $\xi\in \{0,1\}^E$. Let us work in this case and define $\eta_s\sim\ber{s}$ for $s\in\{r,q,r',q'\}$ all independent. We distinguish two cases.
		\begin{itemize}
			\item If $e\in \xi$, note that $r'= 1-(1-r)(1-q')$ and let us compute 
			\begin{align*}
				&\P(e\in (\xi\cap \eta_r)\cup \eta_{q'})=1-(1-r_e)(1-q_e')=r_e',\\
				&\P(e\in (\xi\cup \eta_q)\cap \eta_{r'})=  r'_e.
			\end{align*}
		\end{itemize}
(For the first equation, recall that $q_e'=r_e'q_e$.)
		\begin{itemize}
			\item If $e\notin \xi$, we now have
			\begin{align*}
				&\P(e\in (\xi\cap \eta_r)\cup \eta_{q'})=q_e',\\
				&\P(e\in (\xi\cup \eta_q)\cap \eta_{r'})=  q_er'_e=q_e'.
			\end{align*}
		\end{itemize}
		Since the coordinate laws agree for every $e\in E$ and the coordinates are independent, we conclude.
	\end{proof}
	
\end{proof}

We have just seen that the property of the dual weak domination is stable under taking intersection; not surprisingly, a similar result holds for the almost-equivalent-but-weaker condition that $\mu(\A_F)\ge\nu(\A_F)$ for all $F\subset E$. Such a result will be useful in Section \ref{section weak * below edag}, where we show that this weaker property implies dual weak domination.

\begin{proposition}\label{proposition strict inequalities below large inequalities}
Let $\mu,\nu$ be measures such that $\mu(\A_F)\ge\nu(\A_F)$ for all $F\subset E$. Then $\mu_p(\A_F)\ge\nu_p(\A_F)$ for every parameter $p$. Furthermore, if the inequality is strict for singletons, we have strict inequalities $\mu_p(\A_F) > \nu_p(\A_F)$ for any parameter $p<1$ and non-empty subset $F$; in particular, $(\mu_p)^c$ weakly dominates $(\nu_p)^c$.
\end{proposition}

\begin{proof}
Indeed, upon conditioning on $\omega\sim\ber{p}$ first, one sees that \begin{equation}\label{equation A_F pour mu_p}
\mu_p(\A_F)=\sum_{H\subset F}\ber{p}(F\cap\omega = H)\mu(\A_H).
\end{equation}
An analogous equation holding for $\nu$ allows us to deduce the inequality $\mu_p(\A_F)\ge\nu_p(\A_F)$.
\par For the second part, choose $e\in F$ and use the strict inequality for $\{e\}$. Then, since $p<1$, the probability $\ber{p}(F\cap\omega=\{e\})$ is nonzero and therefore equality cannot hold.
\end{proof}

\subsection{Weak forms of the FKG inequality}
\label{subsection on FKG}

We now turn to weak forms of the FKG inequality. The notions from the previous subsection admit natural analogues in the setting of a single measure satisfying a weak FKG condition, rather than two measures connected by stochastic domination.

\subsubsection{Weak FKG} We start by noting that if $\mu$ satisfies the FKG inequality, then $\mu_p$ still has it (see Proposition \ref{proposition weak FKG}), but the converse is not true. This motivates the analogous definition:

\begin{defi}[Weak FKG inequality]Take $p\in (0,1]^E$. We say that a probability measure $\mu$ satisfies the $p$--weak FKG inequality if $\mu_p$ satisfies the classical FKG inequality, that is to say, for all $f,g$ increasing functions
\begin{align*}
	\mu_p(fg)\geq \mu_p(f)\mu_p(g).
\end{align*}
Again, we say that $\mu$ satisfies the weak FKG inequality if there is some parameter $p\in(0,1]^E$ such that $\mu$ satisfies the $p$--weak FKG inequality.
\end{defi}

In this context, we have the following analogue of Proposition \ref{proposition weak stochastic domination} in the setting of weak FKG.
\begin{proposition}\label{proposition weak FKG}
Let $\mu$ be a measure on $\{0,1\}^E$, and let $p\in[0,1]^E$ be a parameter.  \begin{enumerate}[(i)]
\item if $\mu$ satisfies the FKG inequality, then $\mu_p$ also satisfies the FKG property. 
\item if $\mu$ satisfies the weak FKG, then $\mu(\forall_{F\cup {F'}})\ge\mu(\forall_F)\mu(\forall_{F'})$ for all $F,{F'}\subset E$. 
\item if for any $F,{F'}\subset E$, $\mu(\forall_{F\cup {F'}})\ge\mu(\forall_F)\mu(\forall_{F'})$, and the inequality is strict for any nonempty $F,{F'}$, then $\mu$ has the weak FKG property.
\end{enumerate}
\end{proposition}

The proof shares many similarities with that of Proposition \ref{proposition weak stochastic domination}.

\begin{proof}
We start by proving (i). First, note that the product measure $\mu \otimes \ber{p}$ satisfies the FKG inequality, as it is a product of two measures satisfying it. We conclude by using that the pushforward through an increasing map of a measure that has the FKG property also satisfies the FKG property.

To prove (ii), it is enough to consider disjoint $F,{F'}$: in the general case, replacing $F'$ by $F'\setminus(F\cap F')$ and applying the disjoint case yields a stronger statement. Now, for disjoint sets $F,{F'}$, we have that for all $p\in (0,1]^E$ \begin{align*}
\mu(\forall_{F\cup {F'}})\ge\mu(\forall_F)\mu(\forall_{F'}) \iff \mu_p(\forall_{F\cup {F'}})\ge\mu_p(\forall_F)\mu_p(\forall_{F'}).
\end{align*}
Thus, if there is $p\in (0,1]^E$ such that $\mu_p$ satisfies the FKG property, these inequalities are true for $\mu_p$ whence for $\mu$.

\medskip For the third point, as in the proof of Proposition \ref{proposition weak stochastic domination}, it is enough to show that for any two nontrivial increasing events $A,B$, one has $\mu_p(A\cap B)-\mu_p(A)\mu_p(B)>0$ for some constant parameter $p\in(0,1]$. To this end, we make use of equation \eqref{equation mup(A) for small p}. Let $k,\ell$, and $m$ be the minimum cardinalities of witnesses for $A$, $B$, and $A\cap B$, respectively. Note that $m\leq k+\ell$.

Let us separate by cases. When $m<k+\ell$, we have $\mu_p(A)\mu_p(B)\ll\mu_p(A\cap B)$ for $p \ll 1$. Thus the result is true in this case.

We are left with the case $m=k+\ell$. In this case, let $M_A$, $M_B$, and $M_{A\cap B}$ denote the sets of witnesses of minimal cardinality for $A$, $B$, and $A\cap B$, respectively. Note that the equality implies that for all $F\in M_A$ and ${F'}\in M_B$, we have $F\cap {F'}=\emptyset$. Consequently, the map 
\begin{align*}
	M_A\times M_B\xrightarrow{F,{F'}\mapsto F\cup {F'}}M_{A\cap B}
\end{align*} 
is an injection\footnote{Beware, it is \emph{not} true that it is a bijection in the general case; however, we do not need it.}. Indeed, if $F_1\cup F'_1=F_2\cup F'_2$ with $F_i\in M_A$ and $F'_i\in M_B$, then $F_1\cap F'_2=\emptyset$ implies $F_1\subseteq F_2$; by symmetry $F_1=F_2$, and then $F'_1=F'_2$. We can then write
\begin{align*}
\mu_p(A\cap B) &= p^{m}\sum_{H\in M_{A\cap B}}\mu(\forall_H) + o(p^m)\\
&\ge p^{k+\ell}\sum_{\substack{F \in M_A\\{F'}\in M_B}}\mu(\forall_F\cap \forall_{F'})+o(p^m)\\
&>p^k\sum_{F\in M_A}\mu(\forall_F)\times p^{\ell}\sum_{{F'}\in M_B}\mu(\forall_{F'}) + o(p^m) = \mu_p(A)\mu_p(B) +o(p^m).
\end{align*}
We conclude by using the strict inequality and choosing $p>0$ small enough.
\end{proof}

\begin{remark}\label{remark counter example reciprocal weak FKG}
As explained in Remark \ref{remark counter example reciprocal weak domination} for weak domination, strict inequalities in the third item are needed to ensure weak FKG. Indeed, if $E=\{1,2,3\}$ and $\mu$ is the law of $\ber{1/2}$ conditioned on $\{\omega_1+\omega_2+\omega_3\text{ odd}\}$, then one can check that $\mu$ satisfies the large inequalities in the third item, but for any parameter $p$, $\mu_p(A\cap\forall_{\{3\}})<\mu_p(A)\mu_p(\forall_{\{3\}})$ for the increasing event $A=\{\omega_1+\omega_2>0\}$. 
\end{remark}

%In a  this section, we introduce a notion of stochastic domination that lies between weak and classical stochastic domination. To do this, it is useful to introduce the notion of dual weak stochastic domination.

%\begin{defi}[Dual weak stochastic domination] For probability measures $\mu, \nu$ on $\{0,1\}^E$ and a parameter $q\in[0,1)^E$, we say that $\nu$ dually $q$--weakly dominates $\mu$, a relation denoted $\mu\domi^c_q \nu$, if $\nu^c \domi_{1-q} \mu^c$. If there exists such a parameter $q$\footnote{Note that we ask that $q_e<1$ for every edge $e$}, we say that $\nu$ dually weakly dominates $\nu$. \end{defi}

%Note that this dual notion of domination may be understood as follows. The relation $\mu\domi^{c}_q\nu$ is equivalent to the domination $\mu\cup\ber{q}\domi\nu\cup\ber{q}$. Note that the second point of Proposition \ref{proposition weak stochastic domination} implies the inequalities $\mu(\A_F)\ge\nu(\A_F)$ for all $F\subset E$, and the third point implies that in case of strict inequalities for all $\emptyset\neq F\subset E$, we have dual weak domination.

%We can now define the concept of weak$^\dagger$ stochastic domination.

\subsubsection{Weak$^\dagger$ FKG inequality}
\label{sss.weak_dagger_FKG}

Now, as we did for the monotonicity, we introduce a type of FKG inequality which is weaker than the classical version but stronger than the weaker one. As in Subsection \ref{sss.weak_dagger_stochastic_domination}, it is convenient to introduce first the dual weak FKG property.

\begin{defi}[Dual weak FKG]
For a probability measure $\mu$ on $\{0,1\}^E$ and a parameter $q\in[0,1)^E$, we say that $\mu$ has the dual $q$--weak FKG property if $\mu^c$ has the $(1-q)$--weak FKG property. If there exists such a parameter $q$, $\mu$ is said to have the dual weak FKG property.
\end{defi}

As for dual weak domination, this property can be rewritten in terms of the measure $\mu$: $\mu$ has dual $q$--weak FKG if $\mu\cup\ber{q}$ has FKG, or has dual weak FKG if there exists some $q<1$ for which $\mu\cup\ber{q}$ has FKG.

By Proposition \ref{proposition weak FKG}, this property implies $\mu(\A_F\cap\A_{F'})\ge\mu(\A_F)\mu(\A_{F'})$ for any $F,F'\subset E$ and is implied by the strictness of these inequalities for all non-empty subsets $F,F'$.

In view of this definition and the analogous notions for weak domination, the weak$^\dagger$ FKG property is defined as follows:

\begin{defi}[Weak$^\dagger$ FKG] For parameters $p\in(0,1]^E$ and $q\in[0,1)^E$, we say that a measure $\mu$ satisfies the $p,q$--weak$^\dagger$ FKG inequality if $\mu_p$ and $(\mu^c)_{1-q}$ both satisfy the FKG inequality. We say that $\mu$ has the weak$^\dagger$ FKG property if there are parameters $p\in(0,1]^E, q\in[0,1)^E$ such that $\mu$ satisfies the $p,q$--weak$^\dagger$ FKG inequality.
\end{defi}

Let us check that the weak$^\dagger$ FKG property is stable under intersection with a Bernoulli measure.
\begin{proposition}[Monotonicity of weak$^\dagger$ FKG property]\label{p.monotonicity_FKG} Assume that there are parameters $p\in(0,1]^E$ and $q\in[0,1)^E$ such that $\mu$ satisfies the $p,q$--weak$^\dagger$ FKG inequality. Then, for any parameter $r\in(0,1]^E$, $\mu_r$ has the $p',q'$--weak$^\dagger$ FKG property, for $p'= p/r\wedge 1$ and $q'=qr/(1-q(1-r))$.  
\par In particular, weak$^\dagger$ FKG is stable under intersection with $\ber{r}$.
\end{proposition}
\begin{proof}
	The proof is very similar to the proof of Proposition \ref{proposition weak * domination stable}, and also relies on Claim \ref{claim union intersection}. First, it follows directly that $\mu_r$ has $p'$--weak FKG since this is equivalent to $\mu$ having $p'r$--weak FKG, which is true since $p'r\le p$. On the other hand, to show that $\mu_r^c$ has $(1-q')$--weak FKG, or equivalently $\mu_r\cup\ber{q'}$ has FKG, we can note that by Claim \ref{claim union intersection} this amounts to showing that $(\mu\cup\ber{q})\cap\ber{r'}$ has FKG, where $r'=r/(1-q(1-r))$. However, the FKG property is stable under taking intersection with the Bernoulli measure $\ber{r'}$; hence the result is true by assumption on $\mu$.
\end{proof}

To conclude this section, we prove an additional monotonicity which will be key in Section \ref{section Ising's energy field and fk hat} and more specifically Subsection \ref{section weak * below edag} to understand the critical parameter for weak$^\dagger$ FKG in the context of the energy field of the Ising model. In order to introduce this inequality, recall that if $\mu$ has the weak FKG property, then for all $F,F'\subseteq E$ and all parameters $p\in [0,1]^E$, $\mu_p(\forall_{F\cup F'})\geq \mu_p(\forall_F)\mu_p(\forall_F')$. We would also like to show that this is true for the events of type $\A_{F}$. This may look like a consequence of Proposition \ref{p.monotonicity_FKG}, however it does not follow from it. 
\begin{proposition}\label{proposition weak * FKG stable}
Let $\mu$ be a measure such that $\mu(\A_F\cap\A_{F'})\ge\mu(\A_F)\mu(\A_{F'})$ for all $F,F'\subset E$. Then $\mu_p(\A_F\cap\A_{F'})\ge\mu_p(\A_F)\mu_p(\A_{F'})$ for every parameter $p$. Moreover, if the inequality is strict for singletons $\{e\},\{e'\}$, for any parameter $p<1$ we have strict inequalities $\mu_p(\A_F\cap\A_{F'})>\mu_p(\A_F)\mu_p(\A_{F'})$ for any non-empty $F,F'$; in particular, $\mu_p^c$ has weak FKG.
\end{proposition}

\begin{proof}
As in the proof of part (ii) of Proposition \ref{proposition weak FKG}, we only need to focus on the case where $F,F'$ are disjoint. Using the formula \eqref{equation A_F pour mu_p} for $F\sqcup F'$, we get \begin{align*}
\mu_p(\A_{F\cup F'}) = \sum_{H_1\subset F}\sum_{H_2\subset F'}\ber{p}\big((F\cup F')\cap\omega=H_1\cup H_2\big)\mu(\A_{H_1\cup H_2}). 
\end{align*}
Now, if we use the assumption on $\mu$, and formula \eqref{equation A_F pour mu_p} once again for $F,F'$, we get precisely $\mu_p(\A_F\cap\A_{F'})\ge\mu_p(\A_F)\mu_p(\A_{F'})$.

If the inequality is strict for every pair of singletons, choose $e\in F$ and $e'\in F'$. Since $p<1$, the event that $F\cap\omega=\{e\}$ and $F'\cap\omega=\{e'\}$ has positive Bernoulli probability, and the inequality is therefore strict.
\end{proof}

\section{On Ising's energy field}
\label{section Ising's energy field and fk hat}
The objective of this section is to study weak versions of monotonicity and FKG for the energy field of the Ising model. For a finite graph $G=(V,E)$ with positive coupling constants $J=(J_e)_{e\in E}$, recall that the energy field of the Ising model on $G$ at inverse temperature $J$ is the random variable
    \begin{equation}\label{equation energy field from spin configuration}
        \xi(\sigma) = \Big(\mathbbm{1}(\sigma_{x}=\sigma_y)\Big)_{(xy)\in E},
    \end{equation}
where $\sigma$ is distributed as an Ising model on $G$. We write the law of $\xi$ as $\en_J=\en_J^G$. Alternatively, one can see that for any possible configuration $\upxi$
\begin{equation}\label{eq weight configuration energy field}
\en_J(\upxi)\propto \prod_{e:\ \upxi_e=0}e^{-2J_e}.
\end{equation}
Here, by a possible configuration, we mean that there is $\upsigma$ such that $\upxi= \xi(\upsigma)$. Note that \eqref{eq weight configuration energy field} implies that $\xi$ has the law of a percolation configuration of parameter $p=(p_e)_{e\in E}$ with $p_e/(1-p_e)=e^{2J_e}$, conditioned on the event that the result happens to be the energy field (i.e. the gradient) of some vertex configuration $\upsigma$.

It is known that the measures $\en_{J}$ are in general \emph{not} increasing in any component of $J$ in the sense of stochastic domination \cite{haggstrom1996note}; they also do not always satisfy the FKG inequality \cite{klausen2022monotonicity}\footnote{In this reference it is stated for the loop $O(1)$ model which is the dual measure of the energy field}. The objective of this section is to first show that weak forms of such properties, as defined in Section \ref{section general theory of domination}, are still true. Then, we find the exact thresholds for the parameters $p=p(J)\in [0,1]^E$ for which $(\en_J)_p=\en_J\cap\ber{p}$ is monotone or is weak$^\dagger$ monotone. The same will be done for the FKG property and weak$^\dagger$ FKG. Since we use a lot the measures $(\en_J)_p$ for distinct parameters $J,p$, it is convenient to write \begin{align*}
\en_{J,p}:=(\en_J)_p
\end{align*}
to ease the writing.

\subsection{Results}

In this section, we prove two results regarding the threshold for the monotonicity properties discussed in Section \ref{section general theory of domination}. The first result states that the threshold for monotonicity and FKG is the FK-percolation parameter associated with $J$, namely $p(J):=1-e^{-2J}$.

\begin{theorem}[Threshold for monotonicity and FKG property]\label{theorem weak monotonicity and FKG for energy field}
Let $G=(V,E)$ be a finite graph with coupling constants $J$. Then for any $p\leq p(J)$, $\en_J$ has the $p$--weak FKG property and for any $J_2\geq J_1\geq J$, $\en_{J_1}\domi_{p}\en_{J_2}$.
	
	Furthermore, take $n\in \mathbb N$, $J\in (\R^+)^{n}$, and $p\in (0,1]^{n+1}$ such that 
\begin{align}\label{e.assumption on p}
	\sum_{k=1}^n\frac{p_k-p(J)_k}{p_k}>1.
\end{align} Then, there exists a graph $G$ with $n+1$ edges with coupling constants $J_G$ that satisfies the following:
\begin{itemize}
	\item If we number the edges of $G$, $J_G$ restricted to its first $n$ edges is equal to $J$.
	\item Its energy field does not satisfy the $p$--weak FKG property.
	\item There is $J'_G\geq J_G$ such that $\en_{J'_G}$ does not $p$--weakly dominate $\en_{J_G}$.
\end{itemize}
\end{theorem}

\begin{remark}\label{remark what does the assumption above p(J) mean?}
Assumption \eqref{e.assumption on p} might seem complicated. In reality, it is not very restrictive and justifies the vague statement of Theorem \ref{theorem weak stochastic monotonicity for Ising introduction} that one cannot take $p>p(J)$ such that in general $\en_{J,p}$ is either stochastically increasing in $J$ or has the FKG property. Indeed, if one takes $p,J$ to be constant on the edges, inequality \eqref{e.assumption on p} is satisfied as soon as 
\begin{align*}
	\sharp E-1 >\frac{p}{p-p(J)}.
\end{align*}
In other words, if we take $p$ close to $p(J)$ the ``counterexample graph'' $G$ exists as soon as we allow ourselves enough edges. However, note that we are not ruling out the possibility that for any graph $G$ there is a neighbourhood depending on $G$ of parameters around $p(J)$ such that the FKG property and weak stochastic domination in $J$ still hold.
\end{remark}

\begin{remark}[A remark about weak monotonicity of the energy field]\label{remark weak monotonicity already known}
Although it had not appeared previously in such a form, the fact that the energy field of the Ising model was weakly monotone in the coupling constants was already known: indeed, by Proposition \ref{proposition weak stochastic domination}, it is implied by monotonicity of the family of numbers $\big(\en_J(\forall_F)\big)_{F\subset E}$ in the coupling constants, a fact that can be seen to be a consequence of the classical monotonicity of the correlation functions of the spins, as explained below.

\begin{proposition}[Folklore]\label{proposition weak domination ising energy field}
     For $F\subset E$ non-empty, the probability $\en_J(\forall_F)$ is nondecreasing in (any component of) $J$.
\end{proposition}

(Note that to apply Proposition \ref{proposition weak stochastic domination} we need strict inequalities which are not true for some graphs $G$. However, one can easily reduce the analysis to graphs where strict monotonicity holds and then use it. See Subsection \ref{section weak * below edag} for a similar treatment.) 

\begin{proof}
    Indeed, $\en_J(\forall_F)=\big\langle\prod_{(xy)\in F}\frac12(1+\sigma_x\sigma_y)\big\rangle_J=1/2^{\#F}\sum_A\langle\sigma_A\rangle_J$ where $A$ ranges over some subsets of the vertices $V$, and $\sigma_A$ is a shorthand for $\prod_{x\in A}\sigma_x$. However, each term is well known to be increasing in $J$. 
\end{proof}

Unfortunately, this reasoning does not provide an explicit $p$ for which $p$--weak monotonicity holds, whereas such an infinite-volume statement can be deduced from our theorem.
\end{remark}

Coming back to the proof, as explained in the Introduction, the $p(J)$--weak monotonicity follows from Theorem \ref{main theorem introduction} which in turn is proven in Section \ref{section FK percolation}. The FKG part is not new: it is well known that $\fk_{p(J),2}=\en_J\cap\ber{p(J)}$ has FKG. Thus, the main challenge of this theorem is to find the counterexample graph $G$; this is done in Section \ref{section above p(J)}. The ideas there seem new to us.

\medskip Let us note that Theorem \ref{theorem weak monotonicity and FKG for energy field} can be read as a result on the threshold for FKG and stochastic monotonicity for the measures $\en_{J,p}$. We are also interested in the threshold in $p$ for the measures $\en_{J,p}$ when FKG and stochastic monotonicity are replaced by weak$^\dagger$ FKG and weak$^\dagger$ monotonicity. This new threshold must be at least as large as the previous threshold and the following theorem exhibits its exact value which is 
 \begin{align*}
 p^\dagger=p^\dagger(J):= 1-e^{-4J}=1-(1-p(J))^2>p(J).
 \end{align*}

\begin{theorem}[Threshold for weak$^\dagger$ domination and weak$^\dagger$ FKG] \label{theorem weak *}
	Let $G=(V,E)$ be a finite graph with coupling constants $J$. Then for every parameter $p<p^\dagger(J)$, the measure $\en_{J,p}$ satisfies the weak$^\dagger$ FKG inequality and for any $J_2\geq J_1 \geq J$, we have $\en_{J_1,p}\domi^\dagger\en_{J_2,p}$.
	
	Furthermore, there is a graph $G$ such that, for all $p\geq p^\dagger(J)$, the measure $\en_{J,p}$ does not satisfy the weak$^\dagger$ FKG inequality; moreover, there is $J'\geq J$ such that $\en_{J',p}$ does not weakly$^\dagger$ dominate $\en_{J,p}$.
\end{theorem}

The counterexample graph can actually be chosen to be very simple: one may take $C_4$ the cycle graph with four vertices and edges. Note that if we replaced weak$^\dagger$ by weak in Theorem \ref{theorem weak *}, the threshold parameter is trivially $1$ as this follows from Theorem \ref{theorem weak monotonicity and FKG for energy field}. The fact that this is not the case for the weak$^\dagger$ version means in particular that the energy field does not dually weakly increase, nor does it have dual weak FKG. 
\par Since weak$^\dagger$ domination/FKG amounts to having both weak and dual weak domination/FKG, and we already know the ``primal'' part by Theorem \ref{theorem weak monotonicity and FKG for energy field}, we need only to focus on the dual notions.

\medskip
Let us now comment on the percolation $\edag_J=\en_{J,p^\dagger(J)}$ at the threshold point $p^\dagger(J)$: the fact that at $p=p^\dagger$ the weak$^\dagger$ properties do not hold is not a contradiction: indeed, weak$^\dagger$ domination and weak$^\dagger$ FKG are not closed under weak convergence of measures. However, as a corollary of Theorem \ref{theorem weak *} one must have
\begin{align}
	\label{e.weak * FKG} & \en_{J,p^\dagger}(\A_F \cap \A_{F'}) \geq \en_{J,p^\dagger}(\A_F)\en_{J,p^\dagger}(\A_{F'}) & \forall F,F'\subseteq E,\\
	\label{e.weak * monotonicity} & \en_{J,p^\dagger}(\A_F)\geq \en_{J',p^\dagger}(\A_F)& \forall F\subseteq E,\ J'\geq J
\end{align}
because these properties are implied respectively by dual weak FKG and dual weak monotonicity, and are themselves closed under weak convergence of measures. To prove Theorem \ref{theorem weak *}, we actually first prove the inequalities above \eqref{e.weak * FKG} and \eqref{e.weak * monotonicity}. This is the hardest part of the proof and it is the content of Section \ref{section weaker statements at the level of edag}. Then, to conclude the proof, we show in Section \ref{section weak * below edag} that these inequalities imply weak$^\dagger$ FKG and weak$^\dagger$ monotonicity for parameters $p$ strictly below $p^\dagger$. 

\bigskip\noindent
\textbf{Notations}
In the rest of the section, we use the notation $q=(q_e)_{e\in E}=e^{-2J}$, so that $p(J)=1-q$ and $p^\dagger(J)=1-q^2$. We also write $q_F=\prod_{e\in F}q_e$ for $F\subset E$ a subset of edges, and $Z=Z_J$ for the partition function of the energy field defined by equation \eqref{eq weight configuration energy field}, i.e. \begin{equation}\label{eq partition function in q}
Z= Z_J = \sum_{F\in\F}q_F
\end{equation}
where $\F$ is the subset of ${\mathcal P}(E)$ containing $F$ if and only if there is a spin configuration $\sigma\in\{\pm1\}^V$ such that $F$ is the set of \emph{closed} edges of the percolation $\xi(\sigma)$. Beware that $q$ is a \emph{decreasing} function of $J$. 

\subsection{Upper bound for the weak threshold and end of the proof of Theorem \ref{theorem weak monotonicity and FKG for energy field}.}\label{section above p(J)}

In this section, we prove that $p(J)=1-e^{-2J}$ is indeed the threshold for weak FKG and weak monotonicity\footnote{The proof we present for this proposition was found in a discussion with Gemini Pro version 3.1}, concluding the proof of Theorem \ref{theorem weak monotonicity and FKG for energy field}. Although the precise statement in Theorem \ref{theorem weak monotonicity and FKG for energy field} is a bit complicated, the ``counterexample graph'' for this situation is quite simple as we will choose the cycle graph on $n+1$ vertices again.

We start by proving a version of Russo's formula that is an extension of Theorem 2.46 in \cite{grimmett2006random}. It allows us to deal at once with the monotonicity part and the FKG part.

\begin{lemma}\label{lemma derivative in Je}
Let $A$ be an event that does not depend on edge $e$, let $p\in(0,1]^E$ be a parameter, and let $\omega\sim \en_{J,p}$. Then, 
\begin{equation}\label{eq derivative in Je of enJ cap ber p}
\partial_{J_e}[\en_{J,p}(A)]= \frac{2}{p_e}\cov_{J,p}(\1_A,\omega_e),
\end{equation}
where covariance is taken with respect to the measure $\en_{J,p}$.
\end{lemma}
The proof of this formula follows the classical ideas already present in \cite{grimmett2006random}.
\begin{proof}
Let us first compute $\partial_{J_e}\en_J(A)$. Note that, by summing over possible configurations $\upxi\subseteq E$,
\begin{align*}
	\partial_{J_e}\en_J(A) &= \frac{1}{Z_{J}}\partial_{J_e} \sum_{\upxi \in A} e^{2\sum_{f\in E}\upxi_f J_f} - \frac{1}{Z_{J}}\sum_{\upxi \in A} e^{2\sum_{f\in E}\upxi_f J_f} \frac{\partial_{J_e} Z_J}{Z_J}\\
	&= 2\en_J(\xi_e \1_A) - 2\en_J(A)\en_J(\xi_e) \\
	&= 2\cov_{J,1}(\1_A,\xi_e).
\end{align*}

Now, for $F\subseteq E$ we denote the event
\begin{align*}
	A^F:=\{\upxi\subset E: \upxi\cap F\in A\}.
\end{align*}
Take a percolation $\eta\sim\ber{p}$ independent of $\xi\sim \en_J$ and note that $\en_{J,p}(A)=\ber{p}[\en_J(A^\eta)]$. This implies that 
\begin{equation*}
\partial_{J_e}[\en_{J,p}(A)]=\ber{p}\left[\partial_{J_e}\en_J(A^\eta)\right]=2\ber{p}\big[\en_J(\xi_e\1_{A^\eta})\big] - 2\en_J(\xi_e)\ber{p}[\en_J(A^\eta)].
\end{equation*}
Now, since $A$ does not depend on $e$, the random variables $\eta_e$ and $\xi_e \1_{A^\eta}$ are independent. Hence, 
\begin{equation*}
p_e\ber{p}\big[\en_J(\xi_e\1_{A^\eta})\big] = \ber{p}[\en_J(\eta_e \xi_e \1_{A^\eta})]=\en_{J,p}[\omega_e \1_A].
\end{equation*}
We conclude using $p_e\en_J(\xi_e)=\en_{J,p}(\omega_e)$ and $\ber{p}[\en_J(A^\eta)]=\en_{J,p}(A)$.
\end{proof}

Consider the cycle graph $G=(V,E)=C_{n+1}$ whose edge set is identified with $E=\{0,\ldots,n\}$, and fix $J\in\big(\R_{>0}\big)^E$ and $p\in(0,1]^E$ satisfying Assumption \eqref{e.assumption on p}. Recall that we use the notation $q=1-p(J)$, so that this assumption reads \begin{align*}
	\sum_{e=1}^{n} \frac{q_e+p_e-1}{p_e}>1.
\end{align*}
The numbers $J_0,p_0$ are arbitrary. We wish to prove that $\en_{J,p}$ does not have the FKG property and is not stochastically increasing in the coordinate $J_0$.

We start by establishing a decomposition of $\en_J$ as a linear combination of two Bernoulli measures. Remarkably, one of these measures has a parameter strictly bigger than one, which implies that it is \emph{not a probability measure}. More precisely,

\begin{lemma}\label{l.decomposition en_J in sum of bernoulli}
Take $C_1=\prod_{e=0}^n(1+q_e)$, $C_{-1}=\prod_{e=0}^n(1-q_e)$ and $Z=\sum_{F\in\F}q_F$ where $\F\subset{\mathcal P}(E)$ is the set of subsets of $E$ of even cardinality and $q_F=\prod_{e\in F}q_e$ for $F\subset E$. We have that $C_1+C_{-1}=2Z$ and
\begin{equation}\label{eq decomposition en_J in sum of bernoulli}
\en_J = \frac{1}{2Z}\big(C_1\cdot\ber{\frac1{1+q}}+C_{-1}\cdot\ber{\frac1{1-q}}\big)
\end{equation}
as measures.
\end{lemma}

As stated before the proof, we are abusing notation as $1/(1-q)>1$. In this context, for $r\in\R^E$, we define $\ber{r}$ to be the (signed) measure on $\{0,1\}^E$ where 
\begin{align*}
\ber{r}(\eta)=\prod_e r_e^{\eta_e}(1-r_e)^{1-\eta_e}.
\end{align*}
It is straightforward to see that $\ber{r}$ is a probability measure if and only if $r\in[0,1]^E$ and that $\ber{r}(\{0,1\}^E)=1$ for all $r\in \R^E$.

\begin{proof}
In the case of the cycle graph, $\upxi\in \{0,1\}^E$ is the gradient of a vertex configuration if and only if $\upxi^c=1-\upxi$ is such that $\sum_e \upxi^c_e$ is even. In this context, we see that $\en_J$ can be defined by conditioning $\xi\sim\ber{\frac{1}{1+q}}$ on the event $\{\xi^c\in \F\}$. It also makes sense to define the measure $\en_J'$ to be the law of $\xi \sim\ber{\frac{1}{1+q}}$ conditioned on the event $\{\xi^c\notin\F\}$. In these two cases, the partition functions are, up to multiplication by $C_1^{-1}$, 
\begin{align*}
	Z= \sum_{F\in \F} q_F \qquad\text{ and }\qquad Z':=\sum_{F\not\in\F}q_F,
\end{align*}
respectively.

Let us see that we can decompose the Bernoulli measures by using $\en_J$ and $\en_J'$ as follows: \begin{align} \label{e.ber_as_sum_of_energies}
\ber{\frac{1}{1+q}} = \frac1{C_1}\cdot\big( Z\cdot\en_J+Z'\cdot\en_J') \quad \text{ and } \quad
\ber{\frac1{1-q}} = \frac1{C_{-1}}\cdot\big( Z\cdot\en_J-Z'\cdot\en_J').
\end{align}
Indeed, the first equality is just the formula of total probability. The second one could also be interpreted as a formula of total probability, but as it is not a probability measure we need to prove it by hand. Note that for any configuration $\upxi \in \{0,1\}^E$, $Z\en_J(\upxi)-Z'\en'_J(\upxi)$ is equal to
\begin{align*}
& \prod_e q_e^{1-\upxi_e} \1_{\upxi^c \in \F} -\prod_e q_e^{1-\upxi_e} \1_{\upxi^c \notin \F}=(-1)^{\1_{\{\upxi^c\notin \F\}}}\prod_e q_e^{1-\upxi_e}.
\end{align*}
By summing over all configurations $\upxi\in \{0,1\}^E$,  we see that $Z+Z'=C_1$ and $Z-Z'=C_{-1}$; thus $C_1+C_{-1}=2Z$. Furthermore, by multiplying the two identities in \eqref{e.ber_as_sum_of_energies} by $C_1$ and $C_{-1}$, respectively, and summing them, we obtain \eqref{eq decomposition en_J in sum of bernoulli}.
\end{proof}

Before proving the second part of Theorem \ref{theorem weak monotonicity and FKG for energy field}, let us make some comments regarding signed measures with total mass $1$, as this is not a common object in probability theory but many well-known formulae are equally true in this context. In particular, if we have a signed measure $\nu$ with total mass equal to $1$, and $f,g\in L^2(|\nu|)$, we can define
\begin{align*}
	\cov_\nu(f,g):= \int fg d\nu - \int f d\nu \int g d\nu.
\end{align*}
Then, we can check that classical formulae for the covariance still hold in this setup. For example:
\begin{enumerate}[i)]
	\item If $\nu$ is a product measure and $f$ and $g$ depend on different coordinates, then $\cov_\nu(f,g)=0$.
	\item If $\nu= \gamma \nu_1 + (1-\gamma) \nu_2$, where $\gamma\in \R$ and $\nu_1, \nu_2$ are signed measures with total mass $1$, then
	\begin{align}\label{e.cov}
		\cov_\nu(f,g)=\gamma\cov_{\nu_1}(f,g)+(1-\gamma)\cov_{\nu_2}(f,g)+\gamma(1-\gamma)\cdot\big(\nu_2(f)-\nu_1(f)\big)\cdot\big(\nu_2(g)-\nu_1(g)\big).
	\end{align}
	This is the extension of the law of total covariance.
\end{enumerate}

We now have the tools to show the second part of Theorem \ref{theorem weak monotonicity and FKG for energy field}. The strategy is to construct an increasing event $A$ independent of edge $0$ such that $\cov(\1_A,\omega_0)<0$. This prevents $\en_{J}$ from having the $p$--weak FKG property, and, thanks to Lemma \ref{lemma derivative in Je}, $\en_J$ is not $p$--weak monotone in $J$, precisely the two statements we are aiming for. A key technical tool that we use in the proof concerns the intersections\footnote{Recall the intersection of measures defined at the beginning of Section \ref{section FK percolation}.} of Bernoulli measures: $\ber{r}\cap\ber{\tilde r}=\ber{r\tilde r}$ for any parameters $r,\tilde r \in \R^E$.

\begin{proof}[Proof of the second paragraph of Theorem \ref{theorem weak monotonicity and FKG for energy field}]
Define the event $A:=\{\sum_{e=1}^n\omega_e\ge n-1\}$. Using the notations of Lemma \ref{l.decomposition en_J in sum of bernoulli}, we define $\alpha=C_1/(2Z)$ and $\beta=C_{-1}/(2Z)$ and note that $\alpha+\beta=1$. Now, we use Lemma \ref{l.decomposition en_J in sum of bernoulli} and the fact that the intersection of two Bernoulli measures is a Bernoulli measure to see that \begin{equation}\label{eq decomposition en_J cap ber p}
\en_{J,p}=\alpha\cdot\ber{\frac{p}{1+q}}+\beta\cdot\ber{\frac{p}{1-q}}.
\end{equation}
As $\beta=1-\alpha$, we can use \eqref{e.cov} with $\nu_1= \ber{p/(1+q)}$ and $\nu_2= \ber{p/(1-q)}$, $f=\1_A$ and $g=\omega_0$, to see that
\begin{align*}
	\cov_{J,p}(\1_A,\omega_0)=\alpha \beta \frac{2p_0 q_0}{(1+q_0)(1-q_0)} (\nu_2(A)-\nu_1(A)).
\end{align*}
This is because $\cov_{\nu_1}(f,g)=\cov_{\nu_2}(f,g)=0$ since $\nu_1,\nu_2$ are product measures, and $\nu_2(g)=p_0/(1-q_0)$ and $\nu_1(g)=p_0/(1+q_0)$.
We now prove that $\nu_2(A)<0$. Since the constant multiplying $(\nu_2(A)-\nu_1(A))$ is strictly positive, this is enough to show that $\cov_{J,p}(\1_A,\omega_0)<0$.

To show the negativity of $\nu_2(A)=\ber{p/(1-q)}(A)$, let us use the formula for Bernoulli measures: if $r=p/(1-q)$, \begin{equation}
\ber{r}(A)=
\sum_{e=1}^n\ber{r}\big(A, \omega_e=0\big) + \ber{r}\left(\Big\{\sum_{e=1}^n \omega_e=n\Big\}\right) = \prod_{e=1}^n r_e\cdot\left[\sum_{e=1}^n\frac{1-r_e}{r_e}+1\right].
\end{equation}
Since
\begin{align*}
\frac{1-r_e}{r_e}=-\frac{p_e+q_e-1}{p_e},
\end{align*}
and the sum of these quantities over $e\in \{1,\ldots,n\}$ is $<-1$, we obtain that $\nu_2(A)<0$ and thus $\cov_{J,p}(\1_A,\omega_0)<0$. This gives a counterexample for the FKG inequality, and by Lemma \ref{lemma derivative in Je} we see that $\en_{J,p}(A)$ is not increasing in $J_0$.
\end{proof}

\begin{remark}
The decomposition of the energy field into a sum of Bernoulli measures in Lemma \ref{l.decomposition en_J in sum of bernoulli} can be generalized to any finite graph (not only the cycle) and may be worth studying. However, since we only need it for the special case of the cycle, we do not present it.

It is clear from the decomposition used in the proof that $p(J)=1-q$ is a threshold: it is the largest $p$ for which the decomposition \eqref{eq decomposition en_J cap ber p} becomes a convex combination of \emph{probability distributions}. This property holds true for any finite graph. 
\end{remark}

\subsection{No weak$^\dagger$ domination and FKG at and above the threshold $p^\dagger$}\label{ss. at level pdag}
In this section, we prove that $p^\dagger(J)=1-e^{-4J}$ is an upper bound for the threshold of weak$^\dagger$ FKG and weak$^\dagger$ monotonicity. By Propositions \ref{proposition weak * domination stable} and \ref{proposition weak * FKG stable}, it is enough to work at $p=p^\dagger(J)$\footnote{This counterexample was found during the course of a discussion with ChatGPT 5.5}.

Let $G=C_4$, with edge set $E=\{0,1,2,3\}$, and let $q_e=e^{-2J_e}\in(0,1)$. We write
\begin{align*}
	Z=Z_{\rm even}:=\sum_{\substack{F\subset E:\ |F|\ {\rm even}}}q_F,
	\qquad
	Z_{\rm odd}:=\sum_{\substack{F\subset E:\ |F|\ {\rm odd}}}q_F.
\end{align*}
Take $\mu_J=\edag_J=\en_{J,p^\dagger(J)}$ and fix an arbitrary parameter $r\in[0,1)^E$. Set $s_e=1-r_e>0$ and
\begin{align*}
	\mu_J^r:=\mu_J\cup\ber{r}.
\end{align*}
Since an edge is closed for $\mu_J^r$ if and only if it is closed for $\mu_J$ and is not opened by the independent Bernoulli field, for every $F\subset E$,
\begin{align}\label{eq closed events after union for edag on C4}
	\mu_J^r(\A_F)=s_F\mu_J(\A_F).
\end{align}
Moreover, on $C_4$, the admissible closed sets for the energy field are exactly the even subsets of $E$. Hence, directly from the definition of $\mu_J=\en_J\cap\ber{1-q^2}$\footnote{Equation \eqref{eq formula for edag} is the generalization of this formula to any graph.},
\begin{align}\label{eq edag closed events on C4}
	\mu_J(\A_F)=\frac1Z\sum_{\substack{H\subset E:\ |H|\ {\rm even}}}q_{H}q_{F\setminus H}^2
	= \frac{q_F}{Z}\sum_{\substack{H\subset E:\ |H|\ {\rm even}}}q_{H\Delta F}
	= \begin{cases}
		q_F,& |F|\ {\rm even},\\
		q_F\dfrac{Z_{\rm odd}}{Z},& |F|\ {\rm odd}.
	\end{cases}
\end{align}
We now consider the decreasing events
\begin{align*}
	A:=\A_{\{1,3\}}\cup\A_{\{2,3\}},
	\qquad
	B:=\A_{\{0\}}.
\end{align*}
Using \eqref{eq closed events after union for edag on C4} and \eqref{eq edag closed events on C4},
\begin{align*}
\mu_J^r(A)&=s_3q_3(s_1q_1+s_2q_2)-s_1s_2s_3q_1q_2q_3\frac{Z_{\rm odd}}{Z},\\
\mu_J^r(B)&=s_0q_0\frac{Z_{\rm odd}}{Z},\\
\mu_J^r(A\cap B)&=s_0s_3q_0q_3(s_1q_1+s_2q_2)\frac{Z_{\rm odd}}{Z}
-s_0s_1s_2s_3q_0q_1q_2q_3.
\end{align*}
Therefore, using
\begin{align*}
	Z^2-Z_{\rm odd}^2=(Z+Z_{\rm odd})(Z-Z_{\rm odd})=\prod_{e=0}^3(1-q_e^2),
\end{align*}
we obtain
\begin{align}\label{eq negative covariance union edag C4}
\cov_{\mu_J^r}(\1_A,\1_B)
&=-\frac{s_0s_1s_2s_3q_0q_1q_2q_3}{Z^2}\prod_{e=0}^3(1-q_e^2)<0.
\end{align}
Since $A$ and $B$ are decreasing events, the FKG inequality for $\mu_J^r$ would imply that this covariance is non-negative. Thus, for no parameter $r<1$ does $\edag_J\cup\ber{r}$ satisfy FKG. In particular, $\edag_J$ does not have weak$^\dagger$ FKG.

The same example rules out weak$^\dagger$ monotonicity at the threshold. Indeed, the event $A$ is decreasing, and its probability under $\mu_J^r$ satisfies
\begin{align}\label{eq derivative monotonicity union edag C4}
\partial_{J_0}\mu_J^r(A)
&=\frac{2s_1s_2s_3q_0q_1q_2q_3}{Z^2}\prod_{e=1}^3(1-q_e^2)>0.
\end{align}
To see this, write $Z=A_0+q_0B_0$ and $Z_{\rm odd}=B_0+q_0A_0$, where $A_0$ and $B_0$ are respectively the even and odd sums in the variables $q_1,q_2,q_3$; then
\begin{align*}
	\partial_{J_0}\left(\frac{Z_{\rm odd}}{Z}\right)
	=-\frac{2q_0}{Z^2}(A_0^2-B_0^2)
	=-\frac{2q_0}{Z^2}\prod_{e=1}^3(1-q_e^2),
\end{align*}
and \eqref{eq derivative monotonicity union edag C4} follows from the formula for $\mu_J^r(A)$ above. If the family $J\mapsto \big(\en_J\cap\ber{p^\dagger}\big)\cup\ber{r}$ were stochastically increasing (fixing the value of $p^\dagger$ and changing only the measure $\en_J$), the same would be true for $J\mapsto \edag_J\cup\ber{r}$ (where this time the parameter $p^\dagger=p^\dagger(J)$ of the Bernoulli used for thinning changes with $J$) since $J\mapsto p^\dagger(J)$ is increasing. Hence, the probability of every decreasing event would be non-increasing in each coordinate $J_e$. Equation \eqref{eq derivative monotonicity union edag C4} gives the opposite behaviour. Thus no $r<1$ makes $\edag_J\cup\ber{r}$ increasing in $J$, and consequently the threshold value $p^\dagger$ cannot be included in the weak$^\dagger$ domination statement.

\begin{remark}
As we just proved, the obstruction already appears on the cycle $C_4$. Actually, if the parameter $p$ had been taken strictly above $p^\dagger$, one would have had the negative results even for the triangle $C_3$. Indeed, one can check that the inequalities \eqref{e.weak * FKG} and \eqref{e.weak * monotonicity} would not hold for values of $p$ strictly larger than $p^\dagger$.
\end{remark}

\subsection{Lower bound for the weak$^\dagger$ threshold}
\label{subsection proof of theorem weak *}
We are now in position to deal with the first paragraph of Theorem \ref{theorem weak *}, which is its most difficult part. Let us recall the notation $\edag_J:=\en_{J,p^\dagger(J)}$. This measure plays a key role in the proof, a fact which should not come as a surprise since the proof of Theorem \ref{theorem weak monotonicity and FKG for energy field} strongly relies on the properties of $\en_{J,p(J)}=\fk_{p,2}$.

When $G=(V,E)$ is a planar graph, the proof of this result relies on the high-temperature expansion of the Ising model on the dual graph. However, for general graphs, we require a broader framework. We address this in Section \ref{section abstract Ising}, by introducing the Ising model on an abstract graph. While formalising the full Ising model is not strictly necessary, as the high-temperature expansion alone would suffice, doing so provides a much clearer conceptual picture. As the following section is quite abstract, the reader may skip it in a first reading and assume that the graph is planar.

\begin{remark}\label{r.abstract_Ising}
The fact that we need to generalize the Ising model is reminiscent of the fact that the key coupling in \cite{AHL} (Theorem 2.7) seems at first glance to be related to planarity. Indeed, though they construct their percolation model in the general case (in Definition 2.6), it can also be defined in the special case of a planar graph as the dual of the double random current on the dual graph and some of their proofs become easier in this context. Our construction could allow one to actually make sense of this for any graph. 
\par Furthermore, note that in their equation (2.8), if one only takes the intersection $\xi(\sigma)\cap\eta$ (without the extra $\xi(\tilde\sigma)$), one recovers our model $\edag$.
\end{remark}

\subsubsection{Introduction and basic properties of the Ising model on an abstract graph}
\label{section abstract Ising}
In this subsection, we introduce a general framework for defining an Ising model; the reason for this will become clear later in the proofs.

\begin{defi}Let $E$ be a finite set and $\F\subseteq \mathcal P(E)$. We say that $\F$ is a \emph{high-temperature structure} (HT structure) on $E$ if $(E,\F)$ is a subgroup of $(\mathcal P(E),\Delta)$.
\end{defi}
\par Of course, taking $\F$ to be either $\{\emptyset\}$ or $\mathcal P(E)$ gives an HT structure, but we regard these as trivial examples. For a graph $G=(V,E)$ there are two important HT structures. The first one is the \emph{natural HT structure}, defined by \begin{align*}
\F_G=\{\upxi\subset E: \upxi\text{ is an even\footnote{Here by an even graph, we mean a graph for which every vertex has an even degree.} subgraph of }G\}.
\end{align*} 
This structure can be thought of as the trace of all possible loops on $G$. The second one is the \emph{dual HT structure}, and is based on \eqref{equation energy field from spin configuration}:
\begin{align}\label{e.def_F*}
 \F_G^*:=\{\upxi\subset E: \text{there is $\sigma:V \to \{\pm1\}$ such that } \upxi^c=\xi(\sigma)\}.
\end{align}
(Note that $\upxi$ is the \emph{complement} of the energy field of an Ising configuration.)
Indeed, $\F^*_G$ is an HT structure as $\sigma \mapsto \xi(\sigma)^c$ is a group homomorphism. Furthermore, if $G$ is planar, the natural HT structure of its dual $G^*$ can be identified with $\F^*_G$.

As we stated at the beginning of this section, we want to define an abstract Ising model with edge sets on $E$.
To do that we need to define the ``values'' of the model. This is done through the following definition.
\begin{defi}[Space of spin configurations] Let $E$ be a finite set and $\F$ be an HT structure on $E$. We define the space of spin configurations as
		\begin{equation}
			\V = \V_\F=\mathrm{Hom}(\mathcal P(E)/\F,\{\pm1\}),
	\end{equation}
	where $\{\pm1\}$ is viewed as the group with two elements. 
\end{defi}
In other words, $\V$ can be canonically identified with the set of maps $\sigma:\mathcal{P}(E)\to\{\pm1\}$ such that $\sigma(F\Delta F')=\sigma(F)\sigma(F')$ for any $F,F'\subset E$, and $\sigma(\eta)=1$ if $\eta\in\F$. In fancy words, the set $\V$, seen as a group, is the Pontryagin dual of the abelian group $\mathcal P (E)/\F$.

\medskip The guiding intuition for this definition comes from taking a connected graph $G$ and looking at its natural HT structure $\F_G$. In this case, one can naturally identify $\V$ with the set $\{\pm1\}^V/\{\pm\}$ of the spin configurations on $G$ up to a global flip. Indeed, to $\upsigma\in\{\pm1\}^V$ one can associate the homomorphism $h(\upsigma)$ defined by
\begin{align*}
	\upomega \subseteq E\mapsto [{h}(\upsigma)] (\upomega)=\prod_{e=(xy)\in \upomega}\upsigma_x\upsigma_y.
\end{align*}
Note that if $\upomega\in \F_G$, then $[{h}(\upsigma)] (\upomega)=1$, thus $h(\upsigma)$ can be thought of as a homomorphism from $\mathcal P(E)/\F_G$ to $\{\pm1\}$. Furthermore, as $h(\upsigma)$ only depends on $\upsigma$ up to a global flip, $h$ can be defined as a map from $\{\pm1\}^V/\{\pm\}$ to $\mathrm{Hom}(\mathcal P(E)/\F_G,\{\pm1\})$. In the case of a connected graph, one can check that $h$ is, in fact, a bijection, as one can recover $\upsigma$ by fixing a spanning tree with a root and multiplying through the edges.

We can now define an Ising model on $E$ equipped with an HT structure $\F$.
\begin{defi}[Abstract Ising model] Let $E$ be a finite set, $\F$ be an HT structure on $E$ and positive coupling constants $J=(J_e)_{e\in E}\in (\R^+)^E$. An abstract Ising model on $(E,\F)$ with coupling constants $J$ is the probability distribution on $\V_\F$ defined by
	\begin{equation}\label{eq ising abstrait}
		\is(\upsigma) \propto e^{-H(\upsigma)},
	\end{equation}
	where the Hamiltonian is defined as
	\begin{align*}
	H(\upsigma)=-\sum_{e\in E}J_e\upsigma(e).
	\end{align*}
\end{defi}
Note that if we fix a connected graph $G=(V,E)$ and we define an abstract Ising model on the natural HT structure $\F_G$, then it is precisely the (free) classical Ising model up to a global flip through the identification $h^{-1}$. Furthermore, for $A\subseteq V$, the classical correlation functions $\langle\upsigma_A\rangle_{G,J}$ are interpreted naturally as $\langle\upsigma(F)\rangle$ for $F\subset E$ such that $\partial F=A$. (Such an $F$ exists only when $A$ has even cardinality; however, if it is not the case, the correlation $\langle\upsigma_A\rangle_{G,J}$ vanishes by spin-flip symmetry.)

Let us now write a high-temperature expansion for the abstract Ising model. This is a generalisation of the classical case, which can be found, e.g., in \cite{FV} (Section 3.7.3).
 \begin{proposition}[High-temperature expansion]
Let $\sigma$ be an abstract Ising model on $(E, \F)$ with coupling constants $J$. Then for every $F\subseteq E$
\begin{equation}\label{eq HTE for abstract Ising}
	\big\langle\sigma(F)\big\rangle_J = \frac{\sum_{\eta\in\F}w(\eta\Delta F)}{\sum_{\eta\in\F}w(\eta)},
\end{equation}
where for $F'\subseteq E$
\begin{align*}
	w(F')= \prod_{e\in F'} \tanh(J_e).
\end{align*}
 \end{proposition}
 \begin{proof}
 	Let us compute $Z_J[F]$, the unnormalised numerator corresponding to the correlation indexed by $F$. We define the constant $C=\prod_{e\in E}\cosh(J_e)$ and compute
 	\begin{align*}
 		\sum_{\upsigma\in\V}\upsigma(F)e^{-H(\upsigma)} = &\sum_{\upsigma\in\V}\upsigma(F)\prod_{e\in E}e^{J_e\upsigma(e)}  \\
 		= & C\sum_{\upsigma\in\V}\upsigma(F)\prod_{e\in E}\big(1+\upsigma(e)\tanh(J_e)\big)  \\
 		= & C\sum_{\eta\subset E}w(\eta) \sum_{\upsigma\in\V}\upsigma(F\Delta\eta) \\
 		= &C \, \#\V  \sum_{\eta\in\F\Delta F}w (\eta)= C\#\V \sum_{\eta\in\F}w(\eta\Delta F).
 		\end{align*}
 	Here, in the last equality we used that $\sum_{\upsigma \in \V} \upsigma(\eta')=\#\V\cdot \1_{\{\eta'\in\F\}}$. This is clear for $\eta'\in\F$ since $\upsigma(\eta')=1$ for any $\upsigma\in\V$ by definition. If $\eta'\notin \F$, there is $\upsigma_0\in \V$ such that $\upsigma_0(\eta')=-1$, and, since $\upsigma \mapsto \upsigma_0\upsigma$ is a bijection from $\V$ to itself, we see that $\sum_{\upsigma \in \V} \upsigma(\eta')=\sum_{\upsigma \in \V} \upsigma(\eta')\upsigma_0(\eta')=-\sum_{\upsigma \in \V} \upsigma(\eta')$. Hence, the sum is zero.
 	
 \end{proof}
 
The GKS inequalities are a direct corollary of the high-temperature expansion of the abstract Ising model.
\begin{proposition}[GKS inequalities for the abstract Ising model]\label{proposition griffith abstract}
Let $\sigma$ be an abstract Ising model on $(E,\F)$. Then, for any $F,G\subset E$, one has $\langle\sigma(F)\rangle\ge0$ and $\langle\sigma(F)\sigma(G)\rangle\ge\langle\sigma(F)\rangle\cdot\langle\sigma(G)\rangle$.
\end{proposition}

\begin{proof}
The first inequality follows from \eqref{eq HTE for abstract Ising}; the second one can be proved in exactly the same way as the version for the classical Ising model (see, e.g., Theorem 3.49 in \cite{FV} and the proof following it). 
\end{proof}

\begin{proposition}[Monotonicity in the HT structure]\label{proposition monotonicity in the HT structure}
Let $\F_1,\F_2$ be two HT structures on $E$, with $\F_1\subset\F_2$. Then for any coupling constants $J$ and any subset of edges $F$, one has $\langle\sigma(F)\rangle_J^{\F_1}\le\langle\sigma(F)\rangle_J^{\F_2}$.
\end{proposition} 

\begin{remark}
To fix ideas, take a graph $G_1=(V_1,E)$ and define $G_2=(V_2,E)$ to be constructed by contracting the edges in $F'$. In this case, the natural HT structures can be compared as $\F_{G_1}\subseteq \F_{G_2}$ and thus our inequality corresponds to the fact that contracting edges increases correlation functions. The fact that $\F^*_{G_2}\subset\F^*_{G_1}$ and the subsequent inequality corresponds this time to the fact that on the abstract dual one has ``morally'' less edges by contracting edges on the primal graph.
\end{remark}

\begin{proof}
	Let us start by using the GKS inequality to see that for the HT structure $\F_1$ and any subset of edges $H,F$ one has
	\begin{equation}\label{e.Switching}
		\sum_{\eta\in\F_1}w(\eta\Delta F)\sum_{\eta\in\F_1}w(\eta\Delta H)\le 
		\sum_{\eta\in\F_1}w(\eta)\sum_{\eta\in\F_1}w(\eta\Delta H\Delta F).
	\end{equation}
	Take $\mathcal H=\{H_1,\ldots,H_k\}\subset\mathcal P (E)$ to be a complete set of representatives of the group $\F_2/\F_1$, and sum \eqref{e.Switching} over all $H\in \mathcal H$ to obtain
	\begin{equation*}
		\sum_{\eta\in\F_1}w(\eta\Delta F)\sum_{H\in \mathcal H}\left(\sum_{\eta\in\F_1}w(\eta\Delta H)\right)\le 
		\sum_{\eta\in\F_1}w(\eta)\sum_{H\in \mathcal H}\left(\sum_{\eta\in\F_1}w(\eta\Delta H\Delta F)\right).
		\end{equation*}
	We conclude by using \eqref{eq HTE for abstract Ising} and noting that
	 \begin{equation*}
\langle\sigma(F)\rangle_J^{\F_2} = \frac{\sum_{H\in \mathcal H}\sum_{\eta\in\F_1}w(\eta\Delta H\Delta F)}{\sum_{H\in \mathcal H}\sum_{\eta\in\F_1}w(\eta\Delta H)}.
\end{equation*}
\end{proof}

\subsubsection{The model $\edag$ and the abstract Ising model on the dual structure}

In this section, we show that the percolation model $\edag_J=\en_{J,p^{\dagger}}$ is related to the abstract Ising model generated by the dual HT structure of $G$. Of course, in the case of a planar graph, one can think of $\edag_J$ as a subset of the dual graph.

For $F\subset E$, we show that $\edag_J(\A_F)$ is given by a particularly nice formula. Because of what follows, it is better to consider more generally the measure $\en_{J,p}$ and to specialize to $p=1-q^2=1-e^{-4J}$ later on. We recall that for a subset $F\subseteq E$ and a function $s: E\to \R$, we write $s_F=\prod_{e\in F} s_e$.
\begin{proposition}\label{proposition formula for en_J cap ber(1-p) (AF)}
Let $G=(V,E)$ be a finite graph with coupling constants $J=(J_e)_{e\in E}$, let $F\subset E$ be a subset of edges, and let $p\in[0,1]^E$ be a parameter. Writing $q=e^{-2J}$ and $r=1-p$, we have \begin{equation}\label{eq moebius function of q}
\en_{J,p}(\A_F) = \frac{\sum_{\eta\in\F^*_G}q_\eta \times r_{F\backslash\eta}}{\sum_{\eta\in\F^*_G}q_\eta}.
\end{equation}
\par Furthermore, when $p=p^\dagger=1-q^2$, one has \begin{equation}\label{eq formula for edag}
\edag_J(\A_F)=\en_{J,p^\dagger}(\A_F)= \Big(\prod_{e\in F}q_e\Big)\langle\sigma(F)\rangle_{J^*},
\end{equation}
where $\sigma$ is an abstract Ising model on the dual HT structure $(E,\F_G^*)$, and the coupling constants $J^*$ are dual to $J$, i.e. defined by $\tanh(J^*)=e^{-2J}$.
\end{proposition}

\begin{proof}
Upon conditioning first on $\xi\sim\en_J$, one sees that 
\begin{align*}
	\en_{J,p}(\A_F)=\sum_{\eta\subseteq E}\en_J(E\setminus \eta)r_{F\backslash\eta}.
\end{align*}
Note that \eqref{eq moebius function of q} follows because, by \eqref{e.def_F*}, $\F^*_G$ is the support of $\en_J$ and the weight of $\eta\in\F^*_G$ for $\en_J$ is proportional to $q_\eta$. When $r=q^2$, \eqref{eq formula for edag} follows from the high-temperature expansion \eqref{eq HTE for abstract Ising} and the fact that $q_\eta \times q^2_{F\setminus \eta}=q_F\times q_{\eta\Delta F}$.
\end{proof}

Fix $e\in E$. We are now interested in the function $q_e\mapsto \en_{J,p}(\A_F)$. In fact, we show it is a M\"obius map, i.e., it is of the form $q_e\mapsto \frac{a q_e + b}{c q_e + d}$.
\begin{coro}\label{coro moebius map and derivative}
	In the same context as Proposition \ref{proposition formula for en_J cap ber(1-p) (AF)}, fix $e\in E$. The function $q_e \mapsto \en_{J,p}(\A_F)$ is a M\"obius map. In particular, its derivative is of constant sign, and when $q$ and $p$ are related by the relation $p=1-q^2$, its sign is equal to the sign of \begin{equation}\label{derivative of the probability}
\left(\sum_{\eta\in\F, e\not\in\eta}q_\eta\right)\left(\sum_{\eta\in\F, e\in\eta}q_{\eta\Delta F}\right)- \left(\sum_{\eta\in\F, e\in\eta}q_\eta\right)\left(\sum_{\eta\in\F, e\not\in\eta}q_{\eta\Delta F}\right),
\end{equation}
where $\F$ is a shorthand for the dual HT structure $\F_G^*$ on $G$.
\end{coro}

\begin{proof}
	From \eqref{eq moebius function of q} it is clear that the function is a M\"obius map. For such maps, the derivative is $\frac{ad-bc}{(cq_e+d)^2}$; in particular its sign is constant and is the sign of $q_e(ad-bc)$ when one writes $\en_{J,p}(\A_F)=\frac{aq_e+b}{cq_e+d}$, which is \begin{equation*}
		\left(\sum_{\eta\in\F, e\not\in\eta}q_\eta\right)\cdot\left(\sum_{\eta\in\F, e\in\eta}q_{\eta}r_{F\backslash\eta}\right)- \left(\sum_{\eta\in\F, e\in\eta}q_\eta\right)\cdot\left(\sum_{\eta\in\F, e\not\in\eta}q_{\eta}r_{F\backslash\eta}\right).
	\end{equation*}
In the case where $p=1-q^2$, i.e. $r=q^2$, we can use once again $q_\eta q_{F\backslash\eta}^2=q_Fq_{\eta\Delta F}$ to factor the term above by $q_F>0$ so that the relevant sign is the sign of \eqref{derivative of the probability}.
\end{proof}

\subsubsection{A weaker version of dual weak domination at $p^\dagger$}
\label{section weaker statements at the level of edag}
In this section, we prove the following inequalities, which will be instrumental in proving the dual weak (hence weak$^\dagger$) domination and FKG for $p<p^\dagger$.

\begin{proposition}[Inequalities at level $p^\dagger$]\label{proposition inequalities for edag}
Take $J \in (\R^+)^E$ and $p^\dagger=p^\dagger(J)=1-e^{-4J}$. Then \begin{itemize}
\item \emph{(Almost dual weak FKG)}
For any $F,F'\subseteq E$ \begin{align}\label{eq fkg weak * for edag}
\edag_J(\A_F\cap\A_{F'})\ge \edag_J(\A_F)\edag_J(\A_{F'}).
\end{align}

\item \emph{(Almost dual weak domination)} For any $J'\ge J$ and $F\subset E$, we have
\begin{equation}\label{eq weak * monotonicity for edag}
\en_{J, p^\dagger}(\A_F)\ge\en_{J',p^\dagger}(\A_F).
\end{equation}
\end{itemize}
\end{proposition}

The proofs of these two inequalities rely on the theory of the abstract Ising model. We start by proving \eqref{eq fkg weak * for edag} which is shorter.

\begin{proof}[Proof of \eqref{eq fkg weak * for edag}]
We may assume that $F,F'$ are disjoint: otherwise, replacing $F'$ by $F'\backslash(F\cap F')$ yields a stronger inequality. Using \eqref{eq formula for edag} and the fact that $\A_F\cap\A_{F'}=\A_{F\Delta F'}$, we see that \eqref{eq fkg weak * for edag} is equivalent to $\langle\sigma(F\Delta F')\rangle_{J^*}\ge\langle\sigma(F)\rangle_{J^*}\langle\sigma(F')\rangle_{J^*}$, which follows from the GKS inequality (Proposition \ref{proposition griffith abstract}).
\end{proof}

We now turn to the proof of \eqref{eq weak * monotonicity for edag}. It also relies on Corollary \ref{coro moebius map and derivative}.

\begin{proof}[Proof of \eqref{eq weak * monotonicity for edag}]
It is convenient first to prove the special case in which $J'$ is larger than $J$ on precisely one edge, as stated in the following claim:

\begin{claim}\label{claim.weak monotonicity on one edge}For any fixed $e\in E$, \eqref{eq weak * monotonicity for edag} is true if $J'=J$ on every edge but $e$.
\end{claim}

\begin{proof}[Proof of the claim]
Take $e\in E$, recall the notation $q=e^{-2J}$ and $q'=e^{-2J'}$, and note that, by assumption, $q_{\tilde e}=q'_{\tilde e}$ for all $\tilde e\in E\backslash\{e\}$, and $q_{e}>q_e'$. Set $\tilde p=p^\dagger(J)=1-q^2$ and define $f$ on $[q'_e,q_e]$ by $f(x)=\en_{J(x),\tilde p}(\A_F)$, where the coupling constants $J(x)$ are defined as \begin{equation*}
J(x)_{e'} = \begin{cases} 
\quad J_{e'} & \text{ if } e'\neq e \\
 -\log(x)/2 &\text{ if } e'=e \quad\quad \text{so that }x=\exp(-2J(x)_e).
\end{cases}
\end{equation*}
In terms of this, what we want to show is precisely $f(q_e)\ge f(q'_e)$, which would follow from the fact that the function $x\mapsto f(x)$ is increasing. By Corollary \ref{coro moebius map and derivative}, $f$ is a M\"obius map and hence its derivative has constant sign. Evaluating this sign at $x=q_e$, where $\tilde p=1-q^2$, shows that it is the sign of
 
\begin{equation}\label{inequality to be proven}
\left(\sum_{\eta\in\F, e\not\in\eta}q_\eta\right)\left(\sum_{\eta\in\F, e\in\eta}q_{\eta\Delta F}\right)- \left(\sum_{\eta\in\F, e\in\eta}q_\eta\right)\left(\sum_{\eta\in\F, e\not\in\eta}q_{\eta\Delta F}\right) \ge 0.
\end{equation}

To prove this last inequality, define $\F^e=\{\eta\in\F:\ e\not\in\eta\}$ and note that if $\F^e=\F$, it is trivial as two of the sums are zero. So assume $\F^e\neq \F$ and take $\eta_0\in\F\setminus\F^e$. By the GKS inequality (Proposition \ref{proposition griffith abstract}) for the HT structure $\F^e$ on $E$ with coupling constants $J^*$, we have \begin{align*}
	\langle\sigma(\eta_0\Delta F)\rangle^{\F^e}_{J^*}\ge\langle\sigma(\eta_0)\rangle^{\F^e}_{J^*}\langle\sigma(F)\rangle^{\F^e}_{J^*}.
\end{align*}
This inequality can be developed and expanded using the high-temperature expansion \eqref{eq HTE for abstract Ising} to become
\begin{equation*}
\left(\sum_{\eta\in\F^e}q_\eta\right)\left(\sum_{\eta\in\F^e}q_{\eta\Delta\eta_0\Delta F}\right)- \left(\sum_{\eta\in\F^e}q_{\eta\Delta\eta_0}\right)\left(\sum_{\eta\in\F^e}q_{\eta\Delta F}\right) \ge 0.
\end{equation*}
Since $\F=\F^e\bigsqcup\left(\F^e\Delta\eta_0\right)$ (indeed, configurations containing $e$ belong to $\F^e\Delta\eta_0$, whereas those not containing $e$ belong to $\F^e$), the inequality above is exactly \eqref{inequality to be proven}. This concludes the proof of the claim.
\end{proof}

We now deduce the general statement for coupling constants $J\le J'$ that may differ on several edges from Claim \ref{claim.weak monotonicity on one edge}: one can define a finite sequence $J(1),\ldots,J(n)$ of coupling constants such that $J(1)=J, J(n)=J'$ and, for any $1\le k\le n-1$, $J(k)\le J(k+1)$ but with equality on every edge but one. Writing $p(k)=1-\exp(-4J(k))$, the claim implies that for all $F\subset E$ \begin{equation*}
\en_{J(k),p(k)}(\A_F)\ge \en_{J(k+1),p(k)}(\A_F).
\end{equation*}
However, since $p(k)\ge p(1)$ for any $k$ (as $J\mapsto p^\dagger(J)$ is an increasing function), Proposition \ref{proposition strict inequalities below large inequalities} implies for all subsets $F\subseteq E$ \begin{equation*}
\en_{J(k),p(1)}(\A_F)\ge \en_{J(k+1),p(1)}(\A_F).
\end{equation*}
Hence, \eqref{eq weak * monotonicity for edag} follows by chaining these inequalities. This concludes the proof of Proposition \ref{proposition inequalities for edag}.
\end{proof}

In fact, at $p^\dagger$ there is also a slightly more general version of the FKG-like inequality, as stated in the following proposition. Though this result is not needed elsewhere, we felt that it deserved to be written down. 
\begin{proposition}
	For any $F,F'\subseteq E$,
	\begin{align}\label{eq bonus}
		\edag_J(\A_F\cap\forall_{F'})\le \edag_J(\A_F)\edag_J(\forall_{F'}).
	\end{align}
\end{proposition}
\begin{proof}
	As explained in the proof of \eqref{eq fkg weak * for edag}, it is enough to treat the case of disjoint $F,F'$. The same computation as for \eqref{eq formula for edag} gives the more general identity \begin{equation}\label{eq formula edag general}
		\edag_J(\A_F\cap\forall_{F'}) = q_F\frac{\sum_{\eta\in\F^{F'}}q_{\eta\Delta F}}{\sum_{\eta\in\F}q_\eta}\prod_{e\in F'}(1-q_e^2),
	\end{equation}
where $\F^{F'}=\{\eta\in\F: \eta\cap F'=\emptyset\}$ is a subgroup of $\F$. 
\par By Proposition \ref{proposition monotonicity in the HT structure}, we have \begin{equation}
		\langle\sigma(F)\rangle_{J^*}^{\F^{F'}} \le \langle\sigma(F)\rangle_{J^*}^{\F}
	\end{equation}
where on the left-hand side the HT structure is $\F^{F'}$ and on the right it is $\F$; using the high-temperature expansion \eqref{eq HTE for abstract Ising}, this can be rewritten \begin{equation}
		\frac{\sum_{\eta\in\F^{F'}}q_{\eta\Delta F}}{\sum_{\eta\in\F}q_\eta} \le \frac{\sum_{\eta\in\F}q_{\eta\Delta F}}{\sum_{\eta\in\F}q_\eta}\times \frac{\sum_{\eta\in\F^{F'}}q_{\eta}}{\sum_{\eta\in\F}q_\eta}.
	\end{equation}
	Upon multiplying by $q_F\prod_{e\in F'}(1-q_e^2)$, we get (by the general equation \eqref{eq formula edag general}) on the left $\edag_J(\A_F\cap\forall_{F'})$ and on the right $\edag_J(\A_F\cap\forall_\emptyset)\edag_J(\A_\emptyset\cap\forall_{F'})=\edag_J(\A_F)\edag_J(\forall_{F'})$, which concludes the proof of \eqref{eq bonus}.
\end{proof}

\subsubsection{End of the proof of Theorem \ref{theorem weak *}: weak$^\dagger$ domination and weak$^\dagger$ FKG below the threshold}
\label{section weak * below edag}

Now we come to the end of the proof of Theorem \ref{theorem weak *}, namely weak$^\dagger$ domination below the threshold $p^\dagger=p^\dagger(J)$. Of course, Proposition \ref{proposition inequalities for edag} brings us close to it, and what remains to be done is to show that these inequalities allow us to prove weak$^\dagger$ domination and weak$^\dagger$ FKG strictly below the threshold $p^\dagger$. 

\begin{proposition}
With the notation of Theorem \ref{theorem weak *}, if $p\in[0,1]^E$ is a parameter such that $p<p^\dagger(J)$, then for every $J'\ge J$, \begin{equation}\label{eq weak * monotonicity below edag}
\en_{J,p}\domi^c\en_{J',p}.
\end{equation}
Furthermore, $\en_{J,p}$ has dual weak FKG.
\end{proposition}

Before proving this proposition, which concludes the proof of Theorem \ref{theorem weak *}, let us remark that the case of a disconnected graph reduces to the case of a connected one. The reason for this is that when $G$ is a disconnected graph, all the measures that we are dealing with factorize as independent measures restricted to each connected component of the graph, and if the properties we want to show hold for each of these components, then they hold as well for $G$ itself.
\par The same is true for ``loop edges'', i.e. those edges whose endpoints are the same vertex: the percolation on them is independent of the rest as well. Therefore, without loss of generality, during the proof we assume that $G$ is connected without loop edges.

We make one more reduction, which removes the possible tree part of the graph. Declare two edges $e,e'\in E$ to be equivalent if either $e=e'$ or there exists a cycle of $G$ containing both of them; this is the connected-component relation of the cycle matroid of $G$. Bridges form singleton classes. In the variables $\tau_{xy}=\sigma_x\sigma_y$ (or, equivalently, $\xi_{xy}=(1+\tau_{xy})/2$), the only constraints are the cycle constraints $\prod_{e\in C}\tau_e=1$. Since each cycle is contained in one of the above classes and the Ising energy field weight factorizes over edges, the law of the energy field factorizes as a product over these classes. This factorization is preserved after intersection with Bernoulli percolation, after taking complements, and when the couplings are changed. Therefore it is enough to prove the proposition on each class separately: stochastic domination tensorizes, and the product of FKG measures is FKG. The one-edge bridge classes are immediate Bernoulli cases. We may thus assume in the proof below that every two distinct edges of $G$ lie on a common cycle.

\begin{proof}
Thanks to Proposition \ref{proposition strict inequalities below large inequalities} (respectively, Proposition \ref{proposition weak * FKG stable}), to prove dual weak domination \eqref{eq weak * monotonicity below edag} (respectively, dual weak FKG for $\en_{J,p}$), it is enough to show that Inequality \eqref{eq weak * monotonicity for edag} is strict for $F=\{e\}$ whenever the comparison is nontrivial (respectively, that Inequality \eqref{eq fkg weak * for edag} is strict for two distinct singletons $F=\{e\}$ and $F'=\{e'\}$).

For a singleton $F=\{e\}$, we have $\A_F=(\forall_F)^c$ and hence $\en_{J,p}(\A_F)=1-p_e\en_J(\forall_F)$. Thus, writing $e=(xy)$ and $e'=(x'y')$, the two strict statements that remain are
\begin{align}
\label{eq strict two-point monotonicity}
\langle\sigma_x\sigma_y\rangle_J&<\langle\sigma_x\sigma_y\rangle_{J'},\\
\label{e.strict Griffith}
\langle\sigma_x\sigma_y\sigma_{x'}\sigma_{y'}\rangle_J
-\langle\sigma_x\sigma_y\rangle_J\langle\sigma_{x'}\sigma_{y'}\rangle_J
&=4\cov_J\big(\xi(\sigma)_e,\xi(\sigma)_{e'}\big)>0.
\end{align}
We now justify the strict inequality in \eqref{e.strict Griffith} for distinct edges $e$ and $e'$: it amounts to checking that in the classical proof of large inequality, under the assumptions on $G$ explained before the proof, one may get a strict inequality. Set $A=\{x,y\}$ and $B=\{x',y'\}$. For a current $\mathbf n\in\mathbb N^E$, write
\[
w_J(\mathbf n)=\prod_{f\in E}\frac{J_f^{\mathbf n_f}}{\mathbf n_f!},
\qquad
Z_A^{\rm cur}=\sum_{\partial\mathbf n=A}w_J(\mathbf n),
\]
where $\partial\mathbf n$ is the set of vertices incident to an odd total current. The random-current representation gives $\langle\sigma_A\rangle_J=Z_A^{\rm cur}/Z_\emptyset^{\rm cur}$, and the switching lemma yields
\begin{align*}
&(Z_\emptyset^{\rm cur})^2
\big(\langle\sigma_A\sigma_B\rangle_J-
\langle\sigma_A\rangle_J\langle\sigma_B\rangle_J\big)\\
&\qquad=
\sum_{\substack{\partial\mathbf n_1=A\mathbin{\Delta}B\\
\partial\mathbf n_2=\emptyset}}
w_J(\mathbf n_1)w_J(\mathbf n_2)
\mathbf 1\!\left\{x'\not\longleftrightarrow y'
\text{ in }\operatorname{supp}(\mathbf n_1+\mathbf n_2)\right\};
\end{align*}
see, e.g., \cite{FV}. Every term on the right-hand side is nonnegative. By the reduction above, $e$ and $e'$ lie on a common cycle $C$. Take $\mathbf n_2=0$ and let $\mathbf n_1$ be equal to $1$ on $C\setminus\{e,e'\}$ and to $0$ elsewhere. Since $\partial C=\emptyset$, its source set is
$\partial\mathbf n_1=\partial e\mathbin{\Delta}\partial e'=A\mathbin{\Delta}B$. Moreover, after deleting $e$ and $e'$ from $C$, the two endpoints of $e'$ are disconnected in $\operatorname{supp}(\mathbf n_1+\mathbf n_2)$: one of the two paths between them used $e'$ and the other used $e$. Hence the indicator above equals $1$ for this pair of currents. Its weight is strictly positive because all coupling constants are positive, so \eqref{e.strict Griffith} is strict.

Finally, differentiation with respect to a coupling gives
\[
\partial_{J_{e'}}\langle\sigma_x\sigma_y\rangle_J
=
\langle\sigma_x\sigma_y\sigma_{x'}\sigma_{y'}\rangle_J
-\langle\sigma_x\sigma_y\rangle_J\langle\sigma_{x'}\sigma_{y'}\rangle_J.
\]
For a nontrivial cycle-matroid class on which $J'\neq J$, choose an edge $e'$ with $J'_{e'}>J_{e'}$. If $e\neq e'$, the strict positivity proved above shows that increasing $J_{e'}$ strictly increases the two-point function associated with $e$. If $e=e'$, the derivative is
$1-\langle\sigma_x\sigma_y\rangle_J^2>0$, since the finite-volume Ising measure has full support. Increasing the remaining coupling constants can only increase the two-point function further by the usual GKS monotonicity. This proves \eqref{eq strict two-point monotonicity}.
\end{proof}

\bibliographystyle{alpha}
\bibliography{references}

@article{martineau2025stochastic,
  title={Stochastic domination and lifts of random variables in percolation theory},
  author={Martineau, S{\'e}bastien and Poudevigne, R{\'e}my and Rax, Paul},
  journal={arXiv preprint arXiv:2504.02427},
  year={2025}
}

@article{ray2022finitary,
  title={Finitary codings for gradient models and a new graphical representation for the six-vertex model},
  author={Ray, Gourab and Spinka, Yinon},
  journal={Random Structures \& Algorithms},
  volume={61},
  number={1},
  pages={193--232},
  year={2022},
  publisher={Wiley Online Library}
}

@article{lelli2024mixing,
  title={Mixing time of random walk on dynamical random cluster},
  author={Lelli, Andrea and Stauffer, Alexandre},
  journal={Probability Theory and Related Fields},
  volume={189},
  number={3},
  pages={981--1043},
  year={2024},
  publisher={Springer}
}

@article{broman2006refinements,
  title={Refinements of stochastic domination},
  author={Broman, Erik I and H{\"a}ggstr{\"o}m, Olle and Steif, Jeffrey E},
  journal={Probability theory and related fields},
  volume={136},
  number={4},
  pages={587--603},
  year={2006},
  publisher={Springer}
}

@article{broman2006dynamical,
  title={Dynamical stability of percolation for some interacting particle systems and $\varepsilon$-movability},
  author={Broman, Erik I and Steif, Jeffrey E},
  journal={The Annals of Probability},
  volume={34},
  number={2},
  pages={539--576},
  year={2006}
}

@article{RS,
  title={Characterizations of amenability through stochastic domination and finitary codings},
  author={Ray, Gourab and Spinka, Yinon},
  journal={arXiv preprint arXiv:2304.13784},
  year={2023}
}

@article{liggett2006stochastic,
  title={Stochastic domination: the contact process, {I}sing models and {FKG} measures},
  author={Liggett, Thomas M. and Steif, Jeffrey E.},
  journal={Annales de l'Institut Henri Poincar{\'e}, Probabilit{\'e}s et Statistiques},
  volume={42},
  number={2},
  pages={223--243},
  year={2006},
  doi={10.1016/j.anihpb.2005.04.002}
}

@article{izyurov2024energy,
  title={Energy correlations in the critical {I}sing model on a torus},
  author={Izyurov, Konstantin and Kemppainen, Antti and Tuisku, Petri},
  journal={The Annals of Applied Probability},
  volume={34},
  number={2},
  pages={1699--1729},
  year={2024},
  publisher={Institute of mathematical statistics}
}

@article{garban2025energy,
  title={Energy field of critical {I}sing model and examples of singular fields in {QFT}},
  author={Garban, Christophe and Kupiainen, Antti},
  journal={arXiv preprint arXiv:2502.02554},
  year={2025}
}

@article{HS,
  title={The energy density in the planar {I}sing model},
  author={Hongler, Cl{\'e}ment and Smirnov, Stanislav},
  journal={Acta mathematica},
  volume={211},
  number={2},
  pages={191--225},
  year={2013},
  publisher={Institut Mittag-Leffler}
}

@article{hansen2025generalcouplingisingmodels,
      title={A {G}eneral {C}oupling for {I}sing {M}odels and {B}eyond}, 
      author={Ulrik Thinggaard Hansen and Jianping Jiang and Frederik Ravn Klausen},
      year={2025},
      eprint={2506.10765},
      archivePrefix={arXiv},
      primaryClass={math.PR},
      url={https://arxiv.org/abs/2506.10765}, 
}

@article{grimmett1995comparison,
  title={Comparison and disjoint-occurrence inequalities for random-cluster models},
  author={Grimmett, Geoffrey},
  journal={Journal of Statistical Physics},
  volume={78},
  number={5},
  pages={1311--1324},
  year={1995},
  publisher={Springer}
}

@article{deijfen2024geometric,
  title={Geometric random intersection graphs with general connection probabilities},
  author={Deijfen, Maria and Michielan, Riccardo},
  journal={Journal of Applied Probability},
  volume={61},
  number={4},
  pages={1343--1360},
  year={2024},
  publisher={Cambridge University Press}
}

@article{dario-garban,
  title={Phase transitions for the {XY} model in non-uniformly elliptic and {P}oisson-{V}oronoi environments},
  author={Dario, Paul and Garban, Christophe},
  journal={Communications in Mathematical Physics},
  volume={406},
  number={5},
  pages={101},
  year={2025},
  publisher={Springer}
}

@article{BGJ,
  title={The random-cluster model on the complete graph},
  author={Bollob{\'a}s, B{\'e}la and Grimmett, Geoffrey and Janson, Svante},
  journal={Probability Theory and Related Fields},
  volume={104},
  number={3},
  pages={283--317},
  year={1996},
  publisher={Springer}
}

@article{FK-perco,
  title={On the random-cluster model: {I}. {I}ntroduction and relation to other models},
  author={Fortuin, Cornelius Marius and Kasteleyn, Piet W},
  journal={Physica},
  volume={57},
  number={4},
  pages={536--564},
  year={1972},
  publisher={Elsevier}
}

@article{FKG,
  title={Correlation inequalities on some partially ordered sets},
  author={Fortuin, Cees M and Kasteleyn, Pieter W and Ginibre, Jean},
  journal={Communications in Mathematical Physics},
  volume={22},
  number={2},
  pages={89--103},
  year={1971},
  publisher={Springer}
}

@article{Griffiths,
  title={Correlations in {I}sing ferromagnets. {I}},
  author={Griffiths, Robert B},
  journal={Journal of Mathematical Physics},
  volume={8},
  number={3},
  pages={478--483},
  year={1967},
  publisher={American Institute of Physics}
}

@article{KS,
  title={General {G}riffiths' inequalities on correlations in {I}sing ferromagnets},
  author={Kelly, Douglas G and Sherman, Seymour},
  journal={Journal of Mathematical Physics},
  volume={9},
  number={3},
  pages={466--484},
  year={1968},
  publisher={American Institute of Physics}
}

@article{BFO,
	title={A new bound for the critical point of the {FK} model for $ q< 1$},
	author={Beffara, Vincent and Faipeur, Corentin and Oke, Tejas},
	journal={arXiv preprint arXiv:2512.16486},
	year={2025}
}

@book{FV,
  title={Statistical mechanics of lattice systems: a concrete mathematical introduction},
  author={Friedli, Sacha and Velenik, Yvan},
  year={2017},
  publisher={Cambridge University Press}
}

@book{grimmett2006random,
  title={The random-cluster model},
  author={Grimmett, Geoffrey R},
  year={2006},
  publisher={Springer}
}

@article{holley1974remarks,
  title={Remarks on the {FKG} inequalities},
  author={Holley, Richard},
  journal={Communications in Mathematical Physics},
  volume={36},
  number={3},
  pages={227--231},
  year={1974},
  publisher={Springer}
}

@article{edwards1988generalization,
  title={Generalization of the {F}ortuin-{K}asteleyn-{S}wendsen-{W}ang representation and {M}onte {C}arlo algorithm},
  author={Edwards, Robert G and Sokal, Alan D},
  journal={Physical review D},
  volume={38},
  number={6},
  pages={2009},
  year={1988},
  publisher={APS}
}

@article{ahlswede1978inequality,
  title={An inequality for the weights of two families of sets, their unions and intersections},
  author={Ahlswede, Rudolf and Daykin, David E},
  journal={Zeitschrift f{\"u}r Wahrscheinlichkeitstheorie und verwandte Gebiete},
  volume={43},
  number={3},
  year={1978}
}

@article{GHS,
  title={Concavity of magnetization of an {I}sing ferromagnet in a positive external field},
  author={Griffiths, Robert B and Hurst, Charles A and Sherman, Seymour},
  journal={Journal of Mathematical Physics},
  volume={11},
  number={3},
  pages={790--795},
  year={1970},
  publisher={American Institute of Physics}
}

@article{AHL,
  title={The {I}sing magnetisation field and the {G}aussian free field},
  author={Alcalde L{\'o}pez, Tom{\'a}s  and Heeney, Lorca and Lis, Marcin},
  journal={arXiv preprint arXiv:2602.05886},
  year={2026}
}

@article{klausen2022monotonicity,
 author = {Ravn Klausen, Frederik},
 title = {On monotonicity and couplings of random currents and the loop-{{\(\mathrm{O}(1)\)}}-model},
 fjournal = {ALEA. Latin American Journal of Probability and Mathematical Statistics},
 journal = {ALEA, Lat. Am. J. Probab. Math. Stat.},
 issn = {1980-0436},
 volume = {19},
 number = {1},
 pages = {151--161},
 year = {2022},
 language = {English},
 url = {alea.impa.br/articles/v19/19-07.pdf},
 zbMATH = {7470634},
 Zbl = {1482.82015}
}

@inproceedings{haggstrom1996note,
  title={A note on (non-) monotonicity in temperature for the {I}sing model},
  author={H{\"a}ggstr{\"o}m, O},
  booktitle={Markov Proc. Rel. Fields},
  volume={2},
  pages={529--537},
  year={1996}
}

@article{cammarota1993stochastic,
  title={Stochastic order and monotonicity in temperature for {G}ibbs measures},
  author={Cammarota, Camillo},
  journal={Letters in mathematical physics},
  volume={29},
  number={4},
  pages={287--295},
  year={1993},
  publisher={Springer}
}

@article{severo2024slab,
  title={Slab percolation for the {I}sing model revisited},
  author={Severo, Franco},
  journal={Electronic Communications in Probability},
  volume={29},
  pages={1--11},
  year={2024},
  publisher={The Institute of Mathematical Statistics and the Bernoulli Society}
}

@article{LSS,
	title={Domination by product measures},
	author={Liggett, Thomas M and Schonmann, Roberto H and Stacey, Alan M},
	journal={The Annals of Probability},
	volume={25},
	number={1},
	pages={71--95},
	year={1997},
	publisher={Institute of Mathematical Statistics}
}

\end{document}